\documentclass{amsart}
\usepackage[utf8]{inputenc}
\usepackage{amsmath}
\usepackage{amsfonts}
\usepackage{amssymb}
\usepackage{amsthm}
\usepackage{graphicx}
\usepackage[colorinlistoftodos]{todonotes}
\usepackage[colorlinks=true, allcolors=blue]{hyperref}
\usepackage{enumerate}
\usepackage{bm}
\usepackage{enumitem}

\newtheorem{thm}{Theorem}[section]
\newtheorem{lem}[thm]{Lemma}
\newtheorem{cor}[thm]{Corollary}

\newtheorem{prop}[thm]{Proposition}
\newtheorem{rmk}[thm]{Remark}

\newcommand{\mres}{\mathbin{\vrule height 1.6ex depth 0pt width
0.13ex\vrule height 0.13ex depth 0pt width 1.3ex}}
\makeatletter
\newcommand*{\bdiv}{%
  \nonscript\mskip-\medmuskip\mkern5mu%
  \mathbin{\operator@font div}\penalty900\mkern5mu%
  \nonscript\mskip-\medmuskip
}
\makeatother

\theoremstyle{definition}

\title[Two capillary embedded geodesics on Riemannian $2$-disks]{Min-max construction of two capillary embedded geodesics on Riemannian $2$-disks}
\author{Dongyeong Ko}
\address{Department of Mathematics, Rutgers University - New Brunswick, Piscataway, NJ 08854}
\email{dk954@math.rutgers.edu}
\begin{document}

\maketitle
\begin{abstract}
     In this paper, we prove the existence of two capillary embedded geodesics with a contact angle $\theta \in (0,\pi/2)$ on Riemannian $2$-disks with strictly convex boundary, where the absence of a simple closed geodesic loop based on a point of boundary is given. In particular, our condition contains the cases of Riemannian $2$-disks with strictly convex boundary, nonnegative Gaussian curvature and total geodesic curvature lower bound $\pi$ of the boundary. Moreover, by providing examples, we prove that our total geodesic curvature condition is sharp to admit a capillary embedded geodesic with a contact angle $\theta \in (0,\pi/2)$ under the nonnegative interior Gaussian curvature condition. We also prove the existence of Morse Index $1$ and $2$ capillary embedded geodesics for generic metric under the assumptions above.
\end{abstract}
\section{Introduction}

Capillary minimal submanifolds in a manifold with boundary $(M,\partial M,g)$ are submanifolds with boundary in $\partial M$ which are minimal which intersect $\partial M$ with a constant angle $\theta \in (0,\pi/2)$, which describe the model of the interfaces between incompressible immiscible fluids physically. We refer to the angle $\theta$ as the contact angle. There is a beautiful survey of capillary surfaces in the book of Finn \cite{F}. Recently, the analysis of capillary minimal surface on Riemannian polyhedra played a key role to prove Gromov's geometric comparison theorems in Li \cite{Li} (see also \cite{Li2}). 

In the context of variational methods in geometry, a capillary minimal surface $\Sigma^{2} = \partial_{rel} \Omega$ on a closed manifold with boundary $(M^{3},\partial M^{3},g)$ with a contact angle $\theta$ is defined as a critical point of the capilarity energy functional defined as follows:
\begin{equation*}
    A^{\theta}(\Omega) = |\partial \Omega \mres \mathring{M}| + \cos \theta |\partial \Omega \mres \partial M|.
\end{equation*}
In particular, we consider an existence problem of capillary embedded geodesics on Riemannian $2$-disks $(D^{2},\partial D^{2},g)$.

A search of geodesics by the min-max theory dates back to Birkhoff \cite{B} on the existence of a closed geodesic on Riemannian $2$-spheres. Lusternik-Schnirelmann \cite{LS1} obtained three closed geodesics by finding non-trivial higher parameter families of simple closed curves in the space of embedded curves. Grayson \cite{G} proved the existence of three simple closed geodesics on any Riemannian $2$-spheres by the curve shortening flow.

In the free boundary and capillary setting, there are relatively fewer results. Lusternik-Schnirelmann \cite{LS1} proved that any bounded domain in $\mathbb{R}^{n}$ with a smooth convex boundary contains at least $n$ distinct orthogonal geodesic chords. Bos \cite{Bo} extended the existence results to Riemannian $n$-balls with locally convex boundary. Moreover, he provided an example of a non-convex domain in $\mathbb{R}^{n}$ which does not admit a free boundary embedded geodesic, which shows the necessity of the boundary convexity condition (See Figure 1 of \cite{Bo}). The discrete curve shortening process in the free boundary setting has been developed by Gluck-Ziller \cite{GZ} and Zhou \cite{Z} which produces free boundary immersed geodesics. Moreover, Donato-Montezuma \cite{DM} recently produced a free boundary embedded geodesic or a geodesic loop on nonnegatively curved surfaces with a convex boundary in the one-parameter Almgren-Pitts min-max setting (see also \cite{AMS} and \cite{L} for a few examples of constructions).

In the forthcoming paper of the author \cite{K2}, we prove the existence of two free boundary embedded geodesics with Morse Index bound on surfaces with convex boundary.

There has been progress in obtaining capillary surface via Almgren-Pitts min-max theory. De Masi-De Philippis \cite{DD} constructed a min-max capillary minimal surface on the convex domain in $\mathbb{R}^{3}$. Li-Zhou-Zhu \cite{LZZ} developed min-max theory of constant mean curvature capillary surfaces in $3$-manifolds with smooth boundary. However, Almgren-Pitts min-max procedure only gives rise to geodesic networks rather than the embedded curves on surfaces (see \cite{Pi} and \cite{Ai}). We are forced to apply additional techniques such as curve shortening flow to obtain an embedded geodesic.

Regarding the Morse theory of minimal submanifolds obtained by min-max methods, Marques-Neves (\cite{MN1}, \cite{MN2}) derived upper and lower bound of Morse Index of min-max minimal hypersurfaces on closed manifolds with dimension $3 \le n+1 \le 7$ by developing deformation techniques. The author \cite{K} obtained Morse Index bounds of the simple closed geodesics on Riemannian $2$-spheres (see also \cite{DMMS}). In contrast to the hypersurface case, in the curve setting, the deformation relies on the interpolation, arising from the curve shortening flow, between two families of simple closed curves with length control.

For critical points of perturbed functionals of the length functional on closed surfaces, a few results on the existence of curves with constant geodesic curvature have been established through degree theory and min-max methods. It has been a long-standing conjecture of Novikov \cite{N} that for any metric on a Riemannian $2$-sphere, there exists a closed embedded curve of constant curvature $c$ for every $c > 0$ (whose immersed version is referred as Arnold's conjecture \cite{Ar}). Schneider \cite{Sc} constructed two simple closed magnetic curves on quarter-pinched Riemannian $2$-spheres using degree theory. Recently, Sarnataro-Stryker \cite{SS2} improved the existence result by constructing a closed embedded curve with any prescribed constant curvature on $1/8$-pinched Riemannian $2$-spheres, based on the Almgren-Pitts frameworks by Zhou-Zhu \cite{ZZ2} and Ketover-Liokumovich \cite{KL} (For the existence results of immersed constant curvature curves, see \cite{AB} and \cite{CZ2}).

Additionally, while it is motivating to expect the existence of multiple critical points of perturbed functionals associated with open domains on manifolds, for instance, constant curvature curves and CMC hypersurfaces, there are very few results. The difficulty in constructing multiple critical points of weighted function mainly arises from the topological simplicity of the space of domains compared to the  space of curves (or hypersurfaces, cycles) in many cases. Dey \cite{D} and Mazurowski \cite{M} have made some progress in the multiple existence of CMC hypersurfaces using variational methods. Mazurowski-Zhou \cite{MZ} defined a half-volume spectrum and proved it is achieved by CMC hypersurfaces together with (potentially) minimal hypersurfaces via Allen-Cahn approach, where the construction of such hypersurfaces by Almgren-Pitts min-max remains open.

To our knowledge, our min-max construction is the first of the multi-parameter min-max construction of domains associated to the perturbed functionals. For critical points of unperturbed area functionals, there have been successful programs to construct infinitely many minimal hypersurfaces on closed manifolds based on Almgren-Pitts min-max theory; Marques-Neves \cite{MN} proved the existence of infinitely many closed embedded minimal hypersurfaces in closed manifolds $(M^{n+1},g)$ with $3 \le n+1 \le 7$ satisfying the Frankel property: any two closed embedded minimal hypersurfaces intersect each other, and Song \cite{So} proved the infinite existence in the general metric. Zhou \cite{Z2} proved the infinite existence for generic metric by settling the multiplicity conjecture (See \cite{SWZ} and \cite{Wa} for the free boundary setting).

We consider a multiple existence problem and Morse theory of embedded geodesics with a contact angle $\theta \in (0,\pi/2)$ at boundaries on Riemannian $2$-disks $(D^{2},\partial D^{2},g)$ with a strictly convex boundary whose boundary total (signed) curvature $\int_{\partial D^{2}} \kappa$ is bounded below by $\pi$. A capillary geodesic $\gamma = \partial_{rel}\Omega$ with a contact angle $\theta \in (0,\pi/2)$ is a critical point of the following functional $L^{\theta}$:
\begin{equation*}
    L^{\theta}(\Omega) = |\partial \Omega \mres \mathring {D}^{2} | + \cos \theta |\partial \Omega \mres \partial D^{2}|.
\end{equation*}
One version of our main result is as following:
\begin{thm}
For $\theta \in (0,\pi/2)$, there exists at least 2 capillary embedded geodesics with contact angle $\theta$ on Riemannian disks $(D^{2},\partial D^{2},g)$ with a strictly convex boundary and nonnegative interior Gaussian curvature in $D^{2}$ and a total signed geodesic curvature on boundary satisfies $\int_{\partial D^{2}} \kappa \ge \pi$.
\end{thm} 

We show that the total curvature value $\pi$ is threshold to achieve two capillary embedded geodesics in a general angle $\theta \in (0,\pi/2)$ in the following sense:

\begin{prop} [Sharpness of the curvature condition] For a given $k \in (0,\pi)$, there exists a Riemannian $2$-disk $(D^{2},\partial D^{2},g)$ with a strictly convex boundary, interior nonnegative Gaussian curvature and a total signed curvature $\int_{\partial D^{2}} \kappa = k$ such that there exists $\theta \in (0, \pi/2)$ satisfying the following: The Riemannian disk $(D^{2},\partial D^{2},g)$ does not admit a capillary embedded geodesic with a contact angle $\theta$.
\end{prop}

\begin{rmk} We can also consider Almgren-Pitts min-max construction for the $L^{\theta}$-functional on non-negatively curved surfaces with a convex boundary as in the free boundary version in Donato-Montezuma \cite{DM}. In the capillary case, we may obtain even worse boundary regularity than that of free boundary cases. By analyzing the possible singularity formulation relying on the cut-and-paste arguments originating from Calabi-Cao \cite{CC}, in the one-parameter min-max, we obtain a capillary geodesic network whose density $\theta^{1}(p)$ at a single boundary singular point $p$ is bounded above by $2$, consisting of $4$ half-geodesics locally.   
\end{rmk}
\begin{rmk} It would be interesting to know whether there are two capillary immersed geodesics. We expect that the discrete curve shortening process to obtain a free boundary geodesic in Gluck-Ziller \cite{GZ} and Zhou \cite{Z} can be generalized to the capillary setting.
\end{rmk}
We denote a \emph{critical geodesic lasso} as a simple geodesic loop which is a critical point of the length functional (see Section 2.3 for the precise definition). Our min-max argument to derive Theorem 1.1 works if the min-max construction does not achieve a critical geodesic lasso as a min-max limit. In this context, we formulate a more general version of Theorem 1.1 after assuming the following hypothesis on the nonexistence of geodesic lassos with bounded length (See Section 3.2 for the more precise formulation of the hypothesis). 
 \begin{equation*}
     \text{Min-max widths are not realized by critical geodesic lassos on } (D^{2}, \partial D^{2},g). \tag{$\star$}
 \end{equation*}
\begin{thm}
Suppose $(D^{2},\partial D^{2},g)$ has a strictly convex boundary and satisfies the hypothesis $(\star)$. Then there exist at least two smooth embedded capillary geodesics on $(D^{2},\partial D^{2}, g)$.
\end{thm}
Theorem 1.1 is obtained from Theorem 1.5 by the simple analysis based on Gauss-Bonnet theorem. Moreover, by the same reasoning, Theorem 1.5 covers the cases of Riemannian $2$-disks with a nonpositive Gaussian curvature or with a total curvature of positive curvature part does not exceed $\pi$. A Riemannian metric $g$ is called $\theta$-\emph{bumpy} if every capillary embedded geodesic with a fixed contact angle $\theta \in (0, \pi/2)$ is nondegenerate i.e. there is no capillary embedded geodesic with a fixed contact angle $\theta \in (0, \pi/2)$ that admits a non-trivial Jacobi field.
\begin{cor}
There are at least distinct two capillary embedded geodesics on Riemannian disks $(D^{2},\partial D^{2},g)$ with a strictly convex boundary and endowed a $\theta$-bumpy metric satisfying the hypothesis $(\star)$. Then for $k=1,2$, there exists a capillary embedded geodesic $\gamma_{k}$ and its corresponding domains $\Omega_{k}$ such that $\gamma_{k} = \partial_{rel} (\Omega_{k})$ satisfying $L^{\theta} (\Omega_{1})< L^{\theta} (\Omega_{2})$.
\end{cor}
As in the Morse theoretic characterization in \cite{K} by the author, we obtain the following Morse Index bound of capillary embedded geodesics:
\begin{thm}
Suppose $(D^{2},\partial D^{2}, g)$ is a Riemannian $2$-disk whose boundary is strictly convex and satisfies the hypothesis $(\star)$ endowed with a $\theta$-bumpy metric. Then for each $k=1,2$, there exists a capillary embedded geodesic $\gamma_{k}= \partial_{rel} \Omega_{k}$ with
\begin{equation*}
    index(\gamma_{k}) = k
\end{equation*}
and the $L^{\theta}$-functional of two geodesics satisfy $L^{\theta}(\Omega_{1})<L^{\theta}(\Omega_{2})$.
\end{thm}
We also obtain an Index bound for general metrics as:
\begin{cor}
For a $2$-Riemannian disk $(D^{2},\partial D^{2},g)$ with a strictly convex boundary satisfying the hypothesis $(\star)$, for $k=1,2$, there exists a capillary embedded geodesic $\gamma_{k}$ with
\begin{equation*}
    index(\gamma_{k}) \le k \le index(\gamma_{k})+ nullity (\gamma_{k}).
\end{equation*}
\end{cor}

We outline the proof of our main results. We discuss the parametrization of one and two parameter family of domains first. Roughly speaking, motivated by the fact that the space of (unparametrized) embedded intervals on $D^{2}$ whose endpoints are on $\partial D^{2}$ relative to the space of point curves retracts onto $(M^{2} ,\partial M^{2})$, where $M^{2}$ is a Möbius strip, which is proven by Smale \cite{Sm2} (see also Appendix of \cite{Ha}). Then the space of connected domains whose relative boundary is an interval can be retracted to an annulus $S^{1} \times I$, which is a double cover of the space of curves. The boundary of this annulus consists of the empty set and the whole disk which corresponds to point curves in the space of embedded curves. Moreover, in the viewpoint of Morse theory, the empty set and the whole disk can be regarded as Morse Index zero critical points (see Proposition 3.3). Then by the strong Morse inequality, we can expect the existence of Morse Index $1$ and $2$ critical points. Based on this intuition, it is natural to parametrize domains by an annulus $S^{1} \times I$ to perform $2$-parameter min-max construction. 

To find these critical points we expect and to ensure they are distinct, we consider a $2$-parameter family of domains parametrized by the annulus $S^{1} \times I$ whose two boundary components $S^{1} \times \{ 0 \}$ and $S^{1} \times \{ 1 \}$ correspond to the empty set and the whole disk, respectively. Then we parametrize the endpoints of relative boundaries of domains in the family and add a condition that the endpoints of relative boundaries to have a topological degree one. By applying Lusternik-Schnirelmann type arguments, unless a geodesic lasso is achieved as a min-max limit, we obtain two distinct capillary embedded geodesics with different $L^{\theta}$-functional values or a $S^{1}$-cycle of capillary embedded geodesics.

For the tightening process to accumulate domains whose $L^{\theta}$-functional is close to a critical value toward capillary embedded geodesics, we apply the curve shortening flow with fixed boundary established by Huisken \cite{H} and Edelen \cite{E}, and the standard tightening process by vector fields as in Colding-De Lellis \cite{CD}. This procedure relying on the flow with a Dirichlet boundary condition forces us to obtain a critical geodesic lasso with an arbitary contact angle instead of capillary embedded geodesics with the prescribed contact angle $\theta$. While Langford-Zhu \cite{LZ} recently obtained results on the long-time behavior of the free boundary curve shortening flow on domains in $\mathbb{R}^{2}$, there is no such analysis in general Neumann boundary condition and even it is expected to behave badly compared to the free boundary curve shortening flow.

Furthermore, we obtain a generic Morse Index bound of capillary embedded geodesics produced by the min-max procedure. We refine the interpolation arguments in \cite{K} by the author in the closed geodesics case using curve shortening flow with fixed boundary. Together with the bumpy metric analysis in Appendix C, we obtain capillary embedded geodesics with Morse Index $1$ and $2$ for generic metrics satisfying the hypothesis $(\star)$.

We then construct examples showing that our boundary condition is sharp given a nonnegatively curved surface. We consider an Euclidean cone as a disk with an extremal metric in the sense that the interior curvature concentrates to a vertex of the cone. If a cone angle is smaller than $\pi/6$, we can find a simple geodesic loop based on a vertex on the boundary. This corresponds to a cone whose boundary total curvature is smaller than $\pi$. Then if we consider a geodesic emanating from any boundary point with a contact angle slightly larger than the contact angle of the simple geodesic loop, then we have an self-intersection point. Then by deforming the ``tip part'' of this cone, we obtain example surfaces which do not admit a capillary embedded geodesic for some contact angles. Hence we deduce that our boundary curvature condition in Theorem 1.1 is optimal. This constructive proof is inspired by the example surface illustrated in Donato-Montezuma (See Figure 1 in \cite{DM}). We expect the absence of a capillary geodesic with a certain contact angle in case of the existence of simple geodesic loop and the rotationally symmetric disk by the idea above.

For $k \in (0,\pi)$, we conjecture that on Riemannian disks $(D^{2},\partial D^{2},g)$ with a strictly convex boundary, nonnegative interior Gaussian curvature in $D^{2}$ and a total signed geodesic curvature on boundary satisfies $\int_{\partial D^{2}} \kappa \ge k$, there exists a capillary embedded geodesic with a contact angle $\theta \in (0, \min(k/2,\pi/2))$ based on cone examples in the proof of Proposition 1.2. However, some geodesics do not appear as a min-max limit of our smooth min-max construction since they achieve smaller $L^{\theta}$-functional than the critical value (see Theorem 1.1 of \cite{DM} and the following discussion therein). Moreover, even by the consideration of Almgren-Pitts min-max to construct such capillary geodesics, it seems to be delicate to rule out singular stationary varifolds of $L^{\theta}$-functional discussed in Remark 1.3. The adaption of idea of Mazurowski-Zhou \cite{MZ} to consider boundary length preserving sweepouts may give rise to obtain short capillary embedded geodesics which are not obtained by our method, however we lose a control of a contact angle so there is a trade-off. 

The organization of this paper is as follows. In Section $2$, we introduce basic notions on capillary geodesics and their properties. In Section $3$, we construct families and domains, explain the hypothesis, and establish tightening arguments toward capillary geodesics. In Section $4$, we show the existence of two capillary geodesics on convex disks with a certain condition. In Section $5$, we prove the Morse Index estimate. In Section $6$, we prove that our boundary curvature condition is sharp.

In Appendix A, we prove the compactness theorem of capillary geodesics. In Appendix B, we recall the curve shortening flow with fixed boundary points. In Appendix C, we prove bumpy metric theorems of capillary embedded geodesics.
\section*{Acknowledgments}
The author is grateful to his advisor Daniel Ketover for suggesting this topic and valuable discussions, and giving valuable comments on the draft. The author thanks Otis Chodosh and Tom Ilmanen for providing insights on curve shortening flow with boundary. The author is also thankful to Yevgeny Liokumovich and Xin Zhou for valuable discussions relating to this work.  The author was partially supported by NSF grant DMS-1906385.
\section{Capillary geodesics and basic properties}

Our ambient surface with boundary $(D^{2},\partial D^{2}, g)$ is a Riemannian $2$-disk with strictly convex boundary $\partial D^{2}$. For a compact subset $K \subset D^{2}$, we define the relative interior $int_{D^{2}}(K)$ as an interior with respect to the relative topology on $D^{2}$, and denote a topological interior of $D^{2}$ as $\mathring{D^{2}}$. We denote the \emph{relative boundary} $\partial_{rel}(\Omega)$ as a set of points in $D^{2}$ which are neither in $int_{D^{2}}(\Omega)$ nor $int_{D^{2}}(D^{2} \setminus \Omega)$. We define the space of simple domain $\mathcal{D}(D^{2})$ as
\begin{equation*}
    \mathcal{D}(D^{2}) = \{ \Omega | \partial_{rel}(\Omega) \text{ is a simple curve segment with endpoints on } \partial D^{2} \}.
\end{equation*}
We allow $\partial_{rel}(\Omega)$ to be empty. i.e. $\mathcal{D}(D^{2})$ contains $\emptyset$ and $D^{2}$.

Also we denote that $\mathfrak{X}(D^{2})$ as a set of smooth vector fields on $D^{2}$, and $\mathfrak{X}_{tan}(D^{2})$ as a set of smooth vector fields on $D^{2}$ that are tangent to $\partial D^{2}$ on $\partial D^{2}$. We always assume that $D^{2}$ is contained in some closed surface $M$ such that $D^{2} \subset M$. Let us denote that $B_{r}(p)$ as a geodesic open ball with radius $r$ centered at $p \in D^{2}$, and $\tilde{B}_{r}^{+}(p)$ as a half geodesic open ball with radius $r$ centered at $p \in \partial D^{2}$.

 Let $\gamma$ be a smooth, embedded curve which can be presented as a relative boundary of a connected domain $\Omega$ in $\mathcal{D}(D^{2})$. We define $\nu$ as a unit normal vector field of $\gamma$ pointing outward $\Omega$, and $\overline{\eta}$ as a unit normal vector field of $\partial D^{2}$ which points outward of $D^{2}$. Moreover, we denote $\eta$ as outward pointing unit normal vectors of $\partial \gamma$ in $\gamma$, and $\overline{\nu}$ as outward pointing unit vectors on $\partial \gamma$ in $\partial D^{2}$. 

 We recall the definition of Hausdorff distance  between two nonempty subsets $A$ and $B$ which is given by
\begin{equation*}
    d_{\mathcal{H}}(A,B) = \max \Big\{ \sup_{a \in A} d(a,B), \sup_{b \in B} d(b,A) \Big\},
\end{equation*}
where $d(a,B) = \inf _{b \in B} d(a,b)$ and $d(a,b)$ is an (intrinsic) distance between two points on $(D^{2},\partial D^{2}, g)$. We also denote that $d_{\partial D^{2}}(\cdot,\cdot)$ as an intrinsic distance between two points on $\partial D^{2}$, and $N_{h, \partial D^{2}}(A)$ as a $h$-neighborhood in $\partial D^{2}$ of $A \subset \partial D^{2}$ with the distance function $d_{\partial D^{2}}$.
 
 Let $V_{1}(D^{2})$ be the space of rectifiable $1$-varifolds whose support is in $D^{2}$, and $IV_{1}(D^{2})$ be a space of integer rectifiable $1$-varifolds in $D^{2}$; $G_{1}(D^{2})$ be the Grassmannian line bundle over $D^{2}$. We endow the varifold $F$-metric on $V_{1}(D^{2})$ as of 2.1(19)(20) in \cite{P}.

\subsection{Capillary geodesics} Capillary geodesics are critical points of weighted length functional defined on $\mathcal{D}(D^{2})$. For $\Omega \in \mathcal{D}(D^{2})$, we denote $\partial^{b}\Omega$ and $\partial^{i}\Omega$ by varifolds $\partial \Omega \mres \partial D^{2}$ and $\partial \Omega \mres \mathring{D^{2}}$, respectively.
For a fixed capillary contact angle $\theta \in (0,\pi/2)$, we define a \emph{capillary varifold} $\partial^{\theta} \Omega \in V_{1}(D^{2})$ as
\begin{equation*}
    \partial^{\theta} \Omega = \partial^{i} \Omega + \cos \theta \cdot \partial^{b} \Omega.
\end{equation*}
Let us call $\partial^{i}\Omega$ by \emph{interior varifold} and $\partial^{b}\Omega$ by \emph{boundary varifold}. Then we define a capillary length functional $L^{\theta}(\Omega)$ on $\mathcal{D}(D^{2})$ as
\begin{equation}
    L^{\theta}(\Omega) =|\partial^{\theta}\Omega| =  |\partial^{i}\Omega| + \cos \theta \cdot |\partial^{b}\Omega|.
\end{equation}

Let $p_{1}$ and $p_{2}$ be two endpoints of $\gamma = \partial_{rel}(\Omega)$ on $\partial D^{2}$. Now we define $\gamma$ as a \emph{capillary embedded geodesic with a contact angle} $\theta \in (0,\pi/2)$ if it is an embedded geodesic and $\langle \eta , \overline{\nu} \rangle = - \cos \theta$ at $p_{1}$ and $p_{2}$. By calculations of the first variation (see 16.2 of \cite{S}), the first variation formula of $L^{\theta}$ for $X \in \mathfrak{X}_{tan}(D^{2})$ holds:
\begin{equation}
    \partial^{\theta}\Omega(X) = - \int_{\gamma} X \cdot H ds + \langle  X , (\eta + \cos \theta \overline{\nu} )\rangle|_{p_{1}} + \langle  X , (\eta + \cos \theta \overline{\nu} )\rangle|_{p_{2}},
\end{equation}
where $H$ is a mean curvature vector of $\gamma$. We call a varifold $V \in V_{1}(D^{2})$ as a stationary varifold if $\partial V(X)=0$ for any $X \in \mathfrak{X}_{tan}(D^{2})$. (2) gives the relation between stationary varifold and capillary geodesics:
\begin{prop} 
For $\Omega \in \mathcal{D}(D^{2})$, a capillary varifold $\partial^{\theta}\Omega$ is stationary if and only if $\gamma := \partial_{rel}(\Omega)$ is a capillary geodesic with contact angle $\theta$.
\end{prop}
Moreover, we prove the following lemma which will be useful later (cf. Corollary 1.3 in \cite{LZZ}).
\begin{prop} Let $\mathcal{S}$ be a compact subset of $\mathcal{D}(D^{2})$ in the $F$-distance metric sense of their associated capillary varifolds. Then for every $\epsilon>0$, there exists $\delta>0$ such that for every $\Omega' \in \mathcal{D}(D^{2})$ and $\Omega \in \mathcal{S}$ 
\begin{equation*}
    F(\partial ^{\theta} \Omega',\partial ^{\theta} \Omega) < \delta \Rightarrow \,\, F(\partial ^{i} \Omega',\partial ^{i} \Omega) < \epsilon.
\end{equation*}
\begin{proof}
For a fixed $\Omega \in \mathcal{S}$, it holds since for a sequence $\{ \Omega_{i} \}_{i \in \mathbb{N}}$ of $\Omega_{i} \in \mathcal{D}(D^{2})$, $\lim_{i \rightarrow \infty} F(\partial^{\theta}\Omega_{i}, \partial^{\theta}\Omega)=0$ holds if and only if $\lim_{i \rightarrow \infty} F(\partial^{i}\Omega_{i}, \partial^{i}\Omega)=0$. The compactness property completes the proof. 
\end{proof}
\end{prop}
\subsection{Second variation formula and Morse Index} For $X \in \mathfrak{X}_{tan}(D^{2})$, let us denote $f = \langle X, \nu \rangle$. For a capillary embedded geodesic $\gamma$ which is induced by $\Omega \in \mathcal{D}(D^{2})$ and $f \in C^{\infty}(\gamma)$, by Section 2.1 of \cite{HS}, we have a second variation formula of $L^{\theta}$ as following:
\begin{equation}
    Q(f,f) := \partial^{2} L^{\theta}|_{\Omega}(X,X) = \int_{\gamma} (|\nabla_{\gamma}f|^{2} - K f^{2}) ds - \frac{1}{\sin \theta}(\kappa(p_{1})f^{2}(p_{1}) + \kappa(p_{2})f^{2}(p_{2})),
\end{equation}
where $\kappa$ is a geodesic curvature of $\partial D^{2}$. We call a capillary geodesic $\gamma$ is stable if $Q(f,f)$ is nonnegative for all $f \in C^{\infty}(\gamma)$, and we define Morse index of $\gamma$ to be a maximal dimension of a subspace of $C^{\infty}(\gamma)$ where $Q$ is negative definite. i.e. the number of negative eigenvalues of the stability operator associated with $Q$.
\subsection{Critical simple geodesic lassos} On $(D^{2},\partial D^{2},g)$, we define a (smooth) closed curve $\gamma: [0,L] \rightarrow D^{2}$ to be a \emph{simple geodesic lasso} if $\gamma(0) = \gamma(L) \in \partial D^{2}$ and $\gamma((0,L))$ is an (unit-speed parametrized) embedded geodesic, where $L$ is a length of $|\gamma|$. In particular, we call a simple geodesic lasso $\gamma$ to be a \emph{critical geodesic lasso} if
\begin{equation}
    \langle \gamma'(0) - \gamma'(L), v \rangle = 0
\end{equation}
for any $v \in T_{\gamma(0)}\partial D^{2}$. Note that (4) provides a geometric description of a critical geodesic lasso that $\gamma$ is a geodesic without self-intersection point but at endpoints intersecting with the same angle with tangent vector of $\partial D^{2}$ at the endpoint. 

Let us denote (nonparametrized) a domain $\Omega \subset D^{2}$ such that $\partial_{rel}(\Omega) = \gamma$ for some critical geodesic lasso $\gamma$ by a \emph{crtical lasso domain}. Even though $\Omega$ may not be in the set of simple domains $\mathcal{D}(D^{2})$, we define capillary varifolds of $\Omega$ correspondingly in such cases. Note that $supp(\partial^{b} \Omega) = \emptyset \text{ or } \partial M$ in this case. We also denote $\Lambda_{las,a}$ as a set of critical lasso domains $\Omega$ such that $L^{\theta}(\Omega) \le a$. Note that the first variation formula (2) still can be applied to $\partial^{\theta}\Omega$ after idntifying endpoints. By combining the definition and the first variation formula (2), we obtain the following:
\begin{prop}
If $\Omega$ is a critial lasso domain, then a corresponding capillary varifold $\partial^{\theta} \Omega$ is a stationary varifold.
    \begin{proof}
    Let us denote a critical geodesic lasso and its vertex by $\gamma$ and $p$, respectively. We parametrize $\gamma$ by its length $L$. Then by considering boundary terms of the first variation of $\partial^{\theta} \Omega$ in (2), we obtain the first variation formula of $\partial^{\theta} \Omega$ as
    \begin{equation*}
        \partial^{\theta}\Omega(X) = - \int_{\gamma} X \cdot H ds + \langle  X , -\gamma'(0)+\gamma'(L) \rangle|_{p}.
    \end{equation*}
     Since $\gamma$ is a geodesic, the interior term vanishes and the boundary term vanishes by (4). Hence $\partial^{\theta}\Omega(X) = 0$ for any $X \in \mathcal{X}_{tan}(D^{2})$ and $\partial^{\theta}\Omega$ is a stationary varifold.
    \end{proof}
\end{prop}

\section{Smooth min-max construction of capillary geodesics} In this section, we introduce one and two parameter smooth min-max construction to construct two distinct capillary geodesics (see \cite{CD} for minimal surface version, \cite{CGK} for multi-parameter version for minimal surface, and \cite{DT} for boundary version). 

We also discuss the hypothesis we impose to, and technical parts such as nontriviality of sweepouts and the pull-tight procedure via curve shortening flow with boundary.
\subsection{One and two parameter families}
We take a smooth topology of capillary varifolds we introduced in the last section on $\mathcal{D}(D^{2})$. First we describe a $1$-parameter family of domains. Let us denote $I = [0,1]$ as a unit interval. We call the family of domains $\{ \Phi_{t} \}_{t \in I}$ a $1$-\emph{sweepout} if
\begin{enumerate}
    \item $\Phi_{0} = \emptyset$ and $\Phi_{1} = D^{2}$.
    \item $\Phi_{t}$ converges to $\Phi_{t_{0}}$ in the Hausdorff topology as $t \rightarrow t_{0}$.
    \item $\Phi_{t} \in \mathcal{D}(D^{2})$ for $t \in [0,1]$.
\end{enumerate}

Now we construct a $2$-parameter family of domains. Note that $\partial D^{2}$ is diffeomorphic to a unit circle $S^{1}$ and we fix a counterclockwise orientation on $\partial D^{2}$. Let us define \emph{a pair of endpoints function} $\partial : \mathcal{D}(D^{2}) \setminus \{ \emptyset , D^{2} \} \rightarrow \partial D^{2} \times \partial D^{2}$ as an ordered pair of points which represents the endpoints of $\gamma = \partial_{rel}(\Omega)$ for $\Omega \in \mathcal{D}(D^{2})\setminus \{ \emptyset , D^{2} \}$, where we take the first and second endpoints of $\partial^{b}{\Omega}$ with the order on the counterclockwise orientation on $\partial D^{2}$ as convention. We define $\pi_{1} , \pi_{2}: \partial D^{2} \times \partial D^{2} \rightarrow \partial D^{2}$ as a projection map of $\partial$ to the first and second coordinate, respectively. We define an \emph{endpoint function} $e : \mathcal{D}(D^{2})\setminus \{ \emptyset , D^{2} \} \rightarrow \partial D^{2}$ as
\begin{equation}
    e(\Omega) = \pi_{1} \partial (\Omega), 
\end{equation}
for $\Omega \in \mathcal{D}(D^{2})\setminus \{ \emptyset , D^{2} \}$.

We define the family of domains $\{ \Phi_{s,t} \}_{(s,t) \in S^{1} \times I}$ as a $2$-\emph{sweepout} if
\begin{enumerate}
    \item $\Phi_{s,0} = \emptyset$ and $\Phi_{s,1} = D^{2}$ for any $s \in S^{1}$.
    \item $\Phi_{s,t}$ converges to $\Phi_{s_{0},t_{0}}$ in the Hausdorff topology as $(s,t) \rightarrow (s_{0},t_{0})$.
    \item $\Phi_{s,t} \in \mathcal{D}(D^{2})$ for $(s,t) \in S^{1} \times [0,1]$.
    \item The map $f_{t}: S^{1} \rightarrow \partial D^{2}$ defined by $f_{t}(s) := e(\Phi_{s,t})$ has a mapping degree $1$ for every $t \in (0,1)$.
\end{enumerate}
We denote $\mathcal{P}_{i}$ to be the set of $i$-sweepouts for $i=1,2$.
\begin{rmk}
\begin{enumerate}[wide=0pt, align=left]
\item By the convention to take an order of endpoints and the definition of mapping degree, we obtain that a corresponding second endpoint function in condition (4), namely $\pi_{2} \partial$, also has a mapping degree $1$ for every $t \in (0,1)$.
\item Since the continuity of $\Phi_{s,t}$ in Hausdorff topology implies the continuity of the endpoint function $e$, the mapping degree $1$ property at one point $t \in (0,1)$ in the condition (4) indeed ensures the mapping degree $1$ for whole $t \in (0,1)$.
\item We use the nonzero mapping degree property to prove the nontriviality of $2$-sweepouts, hence fixing a nontrivial mapping degree to define a $2$-sweepout does not affect the two parameter min-max construction.
\end{enumerate}
\end{rmk}
\begin{rmk}
While our $2$-sweepouts of domains do not directly correspond to $2$-parameter families of embedded curves whose boundary points are on $\partial D^{2}$, there is an analogy between two classes. More specifically, certain $2$-sweepouts of domains appear as a double cover of $2$-parameter family of embedded curves with endpoints on $\partial D^{2}$ detecting a second nontrivial (relative) homology in the space of such embedded curves. The covering map is a projection map sending domains to its relative boundary here. 

For instance, on the flat disk $D^{2} = \{ (x,y)\,| \, x^{2}+y^{2}=1 \} \subset \mathbb{R}^{2}$, we denote the lines $\mathcal{L}(a,b,c)$ to be
\begin{equation*}
    \mathcal{L}(a,b,c) : = \{ ax+by=c \} \cap D^{2}
\end{equation*}
and
\begin{equation*}
    \mathcal{L}:= \bigcup_{(a,b,c) \in \mathbb{R}^{3}} \mathcal{L}(a,b,c).
\end{equation*}
Then the $2$-parameter family
\begin{equation*}
    \mathbb{R}P^{2} \rightarrow \mathcal{L}, \, [a,b,c] \mapsto \mathcal{L}(a,b,c)
\end{equation*}
is a $2$-sweepout of curves on $(D^{2},\partial D^{2})$ and its double cover in $\mathcal{D}(D^{2})$ is a $2$-sweepout of domains after proper reparametrization.
\end{rmk}

For $i=1,2$, we define the $i$-\emph{width} $w_{i}(D^{2})$ of sweepouts $\mathcal{P}_{i}$ by
\begin{equation}
    w_{i}(D^{2}) := \inf_{\Phi \in \mathcal{P}_{i}} \sup_{x \in X_{i} } L^{\theta}(\Phi_{x}),
\end{equation}
where $X_{1} = I$ and $X_{2} = S^{1} \times I$. By the definition of $k$-sweepouts, we have $w_{1}(D^{2}) \le w_{2}(D^{2})$.

We denote the sequence of $k$-sweepout $\{ \Phi^{j} \}$ to be a \emph{minimizing sequence} if $\lim_{j \rightarrow \infty} \sup _{x \in X_{i}} L^{\theta}(\Phi^{j}_{x}) = w_{i}$. For a minimizing sequence $\{ \Phi^{j} \}$ and for some sequence of parameters $\{x_{j} \}$, if $\lim_{j \rightarrow \infty} L^{\theta}( \Phi^{j}_{x_{j}} ) = w_{i}$ then we call $\{ \Phi^{j}_{x_{j}} \}$ as a \emph{min-max sequence}. Moreover, we define $\Lambda_{e,w_{k}}$ to be the set of capillary varifolds $\partial^{\theta} \Omega$ such that $L^{\theta}(\Omega)= w_{k}$, and $\partial_{rel} (\Omega)$ is a multiplicity one embedded geodesic whose both endpoints are on $\partial D^{2}$.

We define $\Lambda_{\theta,\infty,a}$ to be the set of stationary capillary varifolds $V$ such that $|V| \le a$  with respect to $\mathfrak{X}_{tan}(D^{2})$. Moreover, we define $\Lambda_{\theta,w_{k}}$ to be a set of capillary varifolds $\partial^{\theta} \Omega$ whose $L^{\theta}$-functional achieves value $w_{k}$ and relative boundary is a multiplicity one embedded geodesic whose both endpoints are on $\partial D^{2}$, and both contact angles at endpoints are $\theta$. Note that $\Lambda_{e,w_{k}} \cap \Lambda_{\theta,\infty,w_{k}}= \Lambda_{\theta,w_{k}}$ if $\Lambda_{las,w_{k}} = \emptyset$. We denote $\Lambda_{\theta,w_{k},j}$ and $\Lambda^{j}_{\theta,w_{k}}$ by subsets of $\Lambda_{\theta,w_{k}}$ consisting of capillary varifolds whose Morse index less than or equal to $j$ or larger than or equal to $j$, respectively. We denote $\mathcal{D}_{\theta,w_{k}}$ as the corresponding set of domains of the set of capillary varifolds $\Lambda_{\theta,w_{k}}$. Let the \emph{critical set} $\Lambda(\{ \Phi_{j} \})$ be a set of stationary capillary varifolds can be obtained by the limit of min-max sequence induced by  $\{ \Phi_{j} \}$.

We can also deduce that $w_{k} > \cos \theta |\partial D^{2}|$ so that both widths are achieved by nontrivial stationary varifolds. Here we follow Section 2.3 of \cite{LZZ} by taking constant curvature $c=0$ there.

\begin{prop} [Theorem 2.10 in \cite{LZZ}]
For a fixed $\theta \in (0,\pi/2)$ and $k=1,2$, $\cos \theta |\partial D^{2}| < w_{k}$.
\end{prop}
\subsection{Hypothesis}
Let $(D^{2},\partial D^{2},g)$ be a Riemannian disk with a strictly convex boundary. By Proposition 2.2, capillary varifolds of critical lasso domains are critical varifolds so that $\Lambda_{las,a} \subseteq \Lambda_{\theta,\infty,a}$ for any $a \in \mathbb{R}_{+}$. Now we assume the following to avoid detecting a critical geodesic lasso appearing as min-max limit capillary varifolds:
\begin{equation*}
    \Lambda_{las,2w_{2}} = \emptyset. \tag{$\star$}
\end{equation*}
Saying in more detail, there is no criticial lasso whose length is bounded by $2 w_{2}$. 
\subsection{Pulling-tight toward embedded capillary geodesics}

In this section, we discuss the pull-tight process to accumulate capillary varifolds whose $L^{\theta}$-functional is close enough to the widths $w_{k}$ toward capillary embedded geodesics. Our pull-tight procedure is two fold. The first procedure is to accumulate toward embedded geodesics, and the second one is to accumulate towards stationary capillary varifolds so that toward capillary embedded geodesics. Moreover, we obtain the Morse index upper bound of limit capillary embedded geodesics.

We sketch the first procedure by applying convergence and regularity of curve shortening flow with fixed endpoints in surfaces with strictly convex boundary (in Appendix B) à la Edelen \cite{E}. This is a boundary analog of Grayson's theorem \cite{G}. We obtain the boundary version of Theorem 2.1 in \cite{K} as follows.

\begin{thm}
Suppose $(D^{2},\partial D^{2},g)$ to be a smooth Riemannian $2$-disk with strictly convex boundary satisfying the hypothesis $(\star)$. For $k=1,2$, for any minimizing sequence $\{ \Phi^{j} \}$ of $k$-sweepouts, there is a deformed minimizing sequence $\{ \hat{\Phi}^{j} \}$ of $\{ \Phi^{j} \}$ satisfying the following property. For any $s>0$, there is some $0<a<w_{k}$ satisfying
    \begin{equation}
        \{ \partial^{\theta}\hat{\Phi}^{j}_{x} \in V_{1}(D^{2}) : L^{\theta}(\hat{\Phi}^{j}_{x}) \ge w_{k}-a \} \subset \bigcup_{\gamma \in \Lambda_{e,w_{k}}} B^{F}_{s}(\gamma)
        \end{equation}
for all sufficiently large $j$, where $B^{F}_{s}(\gamma)$ is a $F$-metric ball with the center $\gamma$. 
\begin{proof}

By the compactness theorem of embedded geodesics with endpoints on boundary (Theorem A.1), we have compactness of $\Lambda_{e,w_{k}}$ and $F$-distance function $F(\cdot,\Lambda_{e,w_{k}})$ is well-defined.

Due to Theorem B.1, the regularity and smoothness of curve shortening flow with fixed boundary condition, curves converges to geodesic segments as time goes to infinity by the curve shortening flow with convex boundary. Note that the length of curves monotonically decrease by flow, and so $L^{\theta}$-functional also monotonically decreases.

For a given relative boundary of domain $\alpha$ whose both endpoints are on $\partial D^{2}$, let us denote $H(t,\alpha)$ as a curve deformed by curve shortening flow with boundary by time $t$. For any time $t \in [0,\infty)$, $H(t,\alpha)$ is an embedded curve if $\alpha$ is embedded. We define the tightened domain $\hat{\Phi}^{j}_{x} \in \mathcal{D}(D^{2})$ by flow by taking a sequence $t_{j} \rightarrow \infty$ and domain $\hat{\Phi}^{j}_{x}$ such that
\begin{align*}\partial^{i} \hat{\Phi}^{j}_{x}  &= H(t_{j},\partial ^{i} \Phi^{j}_{x}) \\ \partial^{b} \hat{\Phi}^{j}_{x}  &= \partial ^{b} \Phi^{j}_{x},
\end{align*}
for $x \in X_{k}$. Note that $\hat{\Phi}^{j}$ is well-defined by the continuity and regularity property of curve shortening flow with boundary of curves with finite length. Let us argue by contradiction.

Suppose there is a sequence of $\{ \hat{\Phi}^{j_{k}}_{x_{j_{k}}} \}_{k \in \mathbb{N}}$ such that $L^{\theta}(\hat{\Phi}^{j_{k}}_{x_{j_{k}}}) \ge w_{k} - a_{j}$ and $F(\partial^{\theta}\hat{\Phi}^{j_{k}}_{x_{j_{k}}},\Lambda_{e, w_{k}}) \ge s$ and $a_{j} \rightarrow 0$. Notice that $L^{\theta}(\hat{\Phi}^{j_{k}}_{x_{j_{k}}}) - L^{\theta}(\Phi^{j_{k}}_{x_{j_{k}}}) \rightarrow 0$ as $k \rightarrow \infty$ by the convergence of the length. By monotone decreasing property of $L^{\theta}$-functional through curve shortening flow with boundary, $\{ \partial^{\theta}\Phi^{j_{k}}_{x_{j_{k}}} \}$ converges to elements of $\Lambda_{e,w_{k}}$ and $\{ \partial^{\theta}\hat{\Phi}^{j_{k}}_{x_{j_{k}}} \}$ as well. This contradicts to our assumption $F(\partial^{\theta}\hat{\Phi}^{j_{k}}_{x_{j_{k}}},\Lambda_{e, w_{k}}) \ge s$. Also note that a crticial geodesic lasso does not appear as a limit geodesic by the hypothesis $(\star)$.
\end{proof}
\end{thm}

Now we can adopt standard pull-tight arguments of Proposition 4.1 in \cite{CD} by modifying the functional from area functional to $L^{\theta}$-functional, and from $\mathfrak{X}(D^{2})$ to $\mathfrak{X}_{tan}(D^{2})$. We have the pull-tight type theorem as follows:

\begin{thm}
Suppose $(D^{2},\partial D^{2},g)$ is a smooth Riemannian $2$-disk with strictly convex boundary satisfying the hypothesis $(\star)$. For $k=1,2$, for any minimizing sequence $\{ \Phi^{j} \}$ of $k$-sweepouts, there is a deformed minimizing sequence $\{\hat{ \Phi}^{j} \}$ of $\{ \Phi^{j} \}$ satisfying the following property. For any $s>0$, there is some $0<a<w_{k}$ satisfying
    \begin{equation*}
        \{ \partial^{\theta}\hat{\Phi}^{j}_{x} \in V_{1}(D^{2}) : L^{\theta}(\hat{\Phi}^{j}_{x}) \ge w_{k}-a \} \subset \bigcup_{\gamma \in \Lambda_{\theta,w_{k}}} B^{F}_{s}(\gamma)
        \end{equation*}
for all sufficiently large $j$, where $B^{F}_{s}(\gamma)$ is a $F$-metric ball with the center $\gamma$. 
\begin{proof}
By Theorem 3.4, for given minimizing sequence $\{ \Phi^{j} \}$, there is a deformed minimizing sequence $\{ \Psi^{j} \}$ such that for any $s>0$ there is some $0<a'<w_{k}$ satisfying (7):
\begin{equation*}
        \{ \partial^{\theta}\Psi^{j}_{x} \in IV_{1}(D^{2}) : L^{\theta}(\Psi^{j}_{x}) \ge w_{k}-a' \} \subset \bigcup_{\gamma \in \Lambda_{e,w_{k}}} B^{F}_{s}(\gamma).
        \end{equation*}

Now we follow Section 4 in \cite{CD}, modify the functional and change $\mathfrak{X}(M)$ to $\mathfrak{X}_{tan}(D^{2})$. We have a compactness of $\Lambda_{\theta,2w_{2}}$ by Corollary A.2. We construct a tightening map as
\begin{equation*}
    H:[0,1] \times \{\Omega \in \mathcal{D}(D^{2})\, | \, L^{\theta}(\Omega) \le 2 w_{2}\} \rightarrow \{\Omega \in \mathcal{D}(D^{2})\, | \, L^{\theta}(\Omega) \le 2 w_{2}\},
\end{equation*}
satisfying
\begin{enumerate}[label=(\roman*)]
    \item $H(0,\Omega) = \Omega$ for all $\Omega$;
    \item $H(t,\Omega) = \Omega$ if $\partial^{\theta}\Omega \in \Lambda_{\theta,\infty,2w_{2}}$
    \item if $\partial^{\theta}\Omega \notin \Lambda_{\theta,\infty,2w_{2}}$, then
    \begin{equation*}
        L^{\theta}(H(1,\Omega)) - L^{\theta}(\Omega) \le - L(F(\partial^{\theta}\Omega,\Lambda_{\theta,\infty,2w_{2}}))<0,
    \end{equation*}
    where $L:[0,\infty) \rightarrow [0,\infty)$ is a continuous function with $L(0)=0$ and $L(t)>0$ when $t>0$.
\end{enumerate}
Let us define $\hat{\Phi}^{j}_{x} := H(1,\Psi^{j}_{x})$. Then the proof of Proposition 4.1 in \cite{CD} gives that a capillary varifold of every min-max sequence $\{ \hat{\Phi}^{j}_{x_{j}} \}$ of $\{ \hat{\Phi}^{j} \}$ converges to a stationary capillary varifold $V$. Moreover, by the continuity of $H$ and the fact that a stationary varifold $V$ is a fixed point of $H$, and the fact that an interior varifold of every min-max sequence of $\{ \Psi^{j}\}$ indeed converges to an embedded geodesic, we have $V \in \Lambda_{\theta,w_{k}}$ and $|V| = w_{k}$. By the hypothesis $(\star)$, a capillary varifold of critical lasso domain does not appear as a min-max limit capillary varifold, and the proof is completed.
\end{proof}
\end{thm}

Combining with the finiteness theorem on $\theta$-bumpy metric, we can apply Deformation Theorem A in \cite{MN1} to rule out capillary embedded geodesics with large Morse index by the applying same arguments in Theorem 2.3 in \cite{K}:
\begin{thm}
Suppose $(D^{2},\partial D^{2},g)$ to be a smooth Riemannian $2$-disk with strictly convex boundary satisfying the hypothesis $(\star)$. For $k=1,2$, for any minimizing sequence $\{ \Phi^{j} \}$ of $k$-sweepouts, there is a deformed minimizing sequence $\{ \hat{\Phi}^{j} \}$ of $\{ \Phi^{j} \}$ satisfying the following property. For any $s>0$, there is some $0<a<w_{k}$ satisfying
    \begin{equation*}
        \{ \partial^{\theta}\hat{\Phi}^{j}_{x} \in V_{1}(D^{2}) : L^{\theta}(\hat{\Phi}^{j}_{x}) \ge w_{k}-a \} \subset \bigcup_{\gamma \in \Lambda_{\theta,w_{k},k}} B^{F}_{s}(\gamma)
        \end{equation*}
for all sufficiently large $j$, where $B^{F}_{s}(\gamma)$ is a $F$-metric ball with the center $\gamma$.
\end{thm}
\section{Existence of multiple capillary geodesics}
In this section, we prove the existence of multiple capillary embedded geodesics on $(D^{2},\partial D^{2},g)$ satisfying the hypothesis $(\star)$ based on our min-max constructions in Section 3. For $k=1,2$, we proved that each $k$-sweepout produces at least one capillary embedded geodesic in Theorem 3.5 as a limit of min-max sequence (cf. Section 6 of \cite{MN} and Section 3 of \cite{Ke}). We now show that either $k$-widths $w_{k}(D^{2})$ are different each other or there is an $S^{1}$-cycle of capillary geodesics by Lusternik-Schnirelmann type arguments. Hence there are at least two capillary embedded geodesics on $(D^{2}, \partial D^{2},g)$.

First of all, we prove that if a capillary varifold of $\Omega' \in \mathcal{D}(D^{2})$ is sufficiently close to a capillary varifold of a given domain $\Omega \in \mathcal{D}(D^{2})$ in an $F$-metric topology then the endpoint function $e(\Omega')$ should be contained in the small neighborhood $N_{h,\partial D^{2}}(supp(\partial^{b} \Omega))$ which is a proper open subset of $\partial D^{2}$. We apply the arguments in the proof of Lemma 3.1 in \cite{K} with the modification of constants on the relation between $F$-distance and Hausdorff distance.
\begin{lem}
Let $\Omega \in \mathcal{D}(D^{2})$ be a given domain. For every $h>0$, there exists $\delta = \delta(h, \Omega)>0$ if
\begin{equation}
    F( \partial^{\theta} \Omega, \partial^{\theta} \Omega') < \delta
\end{equation}
then
\begin{equation}
   e(\Omega')  \in N_{h,\partial D^{2}}(supp(\partial^{b}\Omega)) .
\end{equation}
\begin{proof}
For a given $h>0$, we can find $0<h'<h$ depending on $\Omega$ such that
\begin{equation}
     \partial D^{2} \cap N_{h'}(\partial \Omega) \subset N_{h,\partial D^{2}}(supp(\partial^{b}\Omega)).
\end{equation}

For $p \in \partial \Omega$, note that the density $\theta(p) \ge \cos \theta>0$. We can apply the arguments of the proof to bound Hausdorff distance of Lemma 3.1 in \cite{K} after changing constants depending on the density lower bound. Let us take $\delta = h'^{2}\cos \theta/10$.  Then by (8) and the application of the arguments in the proof of Lemma 3.1 in \cite{K} directly, we have
\begin{equation}
    d_{\mathcal{H}}(\partial \Omega,\partial \Omega')< h'. 
\end{equation}
(11) implies that $e(\Omega')$ is contained in $\partial D^{2} \cap N_{h'}(\partial \Omega)$, and we obtain (9) from (10).
\end{proof}
\end{lem}
We now prove the multiple existence of capillary embedded geodesics.
\begin{thm}
Suppose $(D^{2},\partial D^{2},g)$ satisfies the hypothesis $(\star)$. If $w_{1}(D^{2}) = w_{2}(D^{2})$, then there exist infinitely many distinct smooth embedded capillary geodesics in $(D^{2},\partial D^{2}, g)$.
\begin{proof}
Let us argue by contradiction. Suppose that there are only finitely many capillary embedded geodesics in $(D^{2},\partial D^{2},g)$. Denote the set of domains $\Omega_{i} \in \mathcal{D}(D^{2})$ such that $L^{\theta}(\Omega_{i}) = w_{2}(D^{2})$ and $\partial^{i} \Omega_{i}$ is a capillary embedded geodesic for each $1 \le i \le k$ by $\mathcal{S} = \{ \Omega_{1}, ..., \Omega_{k}\}$, and the corresponding set of capillary varifolds by $\mathcal{V}_{\mathcal{S}}$, namely $\mathcal{V}_{\mathcal{S}} = \{ \partial^{\theta}\Omega_{1}, ..., \partial^{\theta}\Omega_{k}\}$. There exists $s_{0}>0$ such that $B^{F}_{s_{0}}(\partial^{\theta}\Omega_{i})$ are disjoint each other for $1 \le i \le k$. We denote that $\gamma_{i} = \partial^{\theta} \Omega_{i}$.

Also there exists $h_{0}>0$ such that $N_{h_{0},\partial D^{2}}(supp(\partial^{b}\Omega_{i}))$ are proper open subsets of $\partial D^{2}$ for all $1 \le i \le k$. Then we take $\delta_{i} : = \delta (h_{0}, \Omega_{i})$ in Lemma 4.1 for $1 \le i \le k$.

We consider a minimizing sequence of $2$-sweepouts $\{ \Phi^{j} \}_{j \in \mathbb{N}}$ and its tightened minimizing sequence $\{ \hat{\Phi}^{j} \}_{j \in \mathbb{N}}$ by applying Theorem 3.5 so that it satisfies the following property: For any $s>0$, there exists $a>0$ such that
\begin{equation}
            \{ \partial^{\theta}\hat{\Phi}^{j}_{x} \in V_{1}(D^{2}) : L^{\theta}(\hat{\Phi}^{j}_{x}) \ge w_{2}-a \} \subset \bigcup_{\gamma \in \mathcal{V}_{\mathcal{S}}} B^{F}_{s}(\gamma)
\end{equation}
for sufficiently large $j$. We take $0<s<\min(s_{0}/2,\min_{1 \le i \le k} \delta_{i}/2)$. For each $j$, we define open sets $\mathcal{Y}_{j}$ and $\mathcal{Z}_{j}$ to be
\begin{equation*}
    \mathcal{Y}_{j}:= \{ x \in S^{1} \times I \, |\, F(\partial^{\theta}\hat{\Phi}^{j}_{x},\mathcal{V}_{\mathcal{S}})<2s \}
\end{equation*}
and
\begin{equation}
    \mathcal{Z}_{j}:= \{ x \in S^{1} \times I \,|\, F(\partial^{\theta}\hat{\Phi}^{j}_{x},\mathcal{V}_{\mathcal{S}})>s\}.
\end{equation}
Let us divide $\mathcal{Y}_{j}$ into
\begin{equation*}
    \mathcal{Y}_{j,i}= \{ x \in S^{1} \times I \, |\, F(\partial^{\theta}\hat{\Phi}^{j}_{x},\partial^{\theta}\Omega_{i})<2s \}.
\end{equation*}
By the choice of $s$, $\mathcal{Y}_{j} = \cup_{1 \le i \le k} \mathcal{Y}_{j,i}$ and open sets $\mathcal{Y}_{j,i}$ are disjoint each other.

For a fixed $j$, we first claim that there exists a closed curve $\alpha_{j}:S^{1} \rightarrow \mathcal{Y}_{j}$ which is not nullhomotopic in $S^{1} \times I$. Suppose not. Then since all the curves in $\mathcal{Y}_{j}$ are contractible, we can find a continuous path $\beta_{j}:[0,1] \rightarrow (S^{1} \times I) \setminus \mathcal{Y}_{j} \subset \mathcal{Z}_{j}$ such that $\beta_{j}(0) \in S^{1} \times \{ 0 \}$ and $\beta_{j}(1) \in S^{1} \times \{ 1 \}$ (See the proof of Theorem 1.6 in \cite{Ke} for a similar Lusternik-Schnirelmann type argument).

We have $\hat{\Phi}^{j}_{\beta_{j}(0)} = \emptyset$ and $\hat{\Phi}^{j}_{\beta_{j}(1)} = D^{2}$ and $\Psi(t): = \hat{\Phi}^{j}_{\beta_{j}(t)}$ is a $1$-sweepout. On the other hand, (12) and (13) gives \begin{equation*}
    L^{\theta}(\hat{\Phi}^{j}_{\beta_{j}(t)})<w_{2}(D^{2})-a = w_{1}(D^{2})-a
\end{equation*}
for $t \in I$. This contradicts to the definition of $1$-sweepout and we have an existence of a closed curve $\alpha_{j}$ which is not nullhomotopic in $S^{1} \times I$. 

Then by the fact of $\alpha_{j}$ being not nullhomotopic in $S^{1} \times I$ and the definition (4) of $2$-sweepout, $f_{j} : = e \circ \alpha_{j}: S^{1} \rightarrow \partial D^{2}$, the endpoint function of $\alpha_{j}$, is a continuous map and has a nonzero mapping degree. By the degree theory argument, $f_{j}$ is a surjective map. Since $\alpha_{j}$ is a continuous loop and $\mathcal{Y}_{j,i}$ are disjoint open sets, $\alpha_{j}(S^{1})$ is contained in a single $\mathcal{Y}_{j,i}$ for some $1 \le i \le k$ and $F(\partial^{\theta}\hat{\Phi}^{j}_{\alpha_{j}(t)},\gamma_{i} )<2s$ for $t \in S^{1}$.

By our choice of $s$, we have
\begin{equation*}
    F(\partial^{\theta}\hat{\Phi}^{j}_{\alpha_{j}(t)},\gamma_{i} )<2s< \delta_{i}
\end{equation*}
and we apply Lemma 4.1 and obtain the following:
\begin{equation*}
    f_{j}(t) \in N_{h_{0},\partial D^{2}}(supp(\partial^{b}\Omega_{i}))
\end{equation*}
for $t \in  S^{1}$. Since $N_{h_{0},\partial D^{2}}(supp(\partial^{b}\Omega_{i}))$ is a proper open subset of $\partial M$, we have
\begin{equation*}
    f_{j}(S^{1}) \subseteq N_{h_{0},\partial D^{2}}(supp(\partial^{b}\Omega_{i})) \subsetneq \partial D^{2}.
\end{equation*}
This contradicts to the surjectivity of $f_{j}$. Thus if $w_{1}(D^{2}) = w_{2}(D^{2})$, then there are infinitely many distinct smooth embedded capillary geodesics in $(D^{2}, \partial D^{2}, g)$.
\end{proof}
\end{thm}
For $(D^{2}, \partial D^{2}, g)$ endowed with a $\theta$-bumpy metric, there is no capillary embedded geodesic with a nontrivial Jacobi field. This implies a finiteness of capillary embedded geodesics with an $L^{\theta}$-functional bound $w_{2}(D^{2})$ and we obtain two distinct capillary geodesics with different $L^{\theta}$-functional values. We discuss the genericness of $\theta$-bumpy metrics of capillary embedded geodesics in Appendix C. 

\begin{cor}
There are at least distinct two capillary embedded geodesics on Riemannian disks $(D^{2},\partial D^{2},g)$ with a $\theta$-bumpy metric satisfying the hypothesis $(\star)$. In particular, if $(D^{2},\partial D^{2},g)$ is endowed with a $\theta$-bumpy metric, then for $k=1,2$, there exists a capillary embedded geodesic $\gamma_{k}$ and its corresponding domains $\Omega_{k}$ such that $\gamma_{k} = \partial_{rel} (\Omega_{k})$ and
\begin{equation*}
    w_{k}(D^{2}) = L^{\theta} (\Omega_{k}),
\end{equation*}
their widths satisfy $w_{1}(D^{2})<w_{2}(D^{2})$.
\begin{proof}
By Theorem 3.5, for $k=1,2$, there exists a capillary embedded geodesic $\gamma_{k}$ which is obtained as a min-max limit of a tightened minimizing sequence $\{ \hat{\Phi}^{j} \}$. If $w_{1}(D^{2})<w_{2}(D^{2})$, then $\gamma_{1}$ and $\gamma_{2}$ are distinct. Suppose $w_{1}(D^{2})=w_{2}(D^{2})$. Then by Theorem 4.2, there are infinitely many capillary embedded geodesics whose $L^{\theta}$-functional value is $w_{1}(D^{2})=w_{2}(D^{2})$.

Now we consider the bumpy metric case. Suppose $w_{1}(D^{2})=w_{2}(D^{2})$. Then there is an infinite sequence of capillary embedded geodesics $\{ \gamma_{k} \}$ where their corresponding domains are $\{ \Omega_{k} \}$. Denote the limit capillary geodesic and limit domain by $\gamma$ and $\Omega$, respectively. By compactness theorem of capillary embedded geodesics (Theorem A.1), this induces a nontrivial Jacobi field of the limit capillary varifold $\partial^{\theta}\Omega$ and this contradicts to the $\theta$-bumpy assumption. Hence $w_{1}(D^{2})<w_{2}(D^{2})$ and two capillary geodesics are distinct each other.
\end{proof}
\end{cor}

We now prove Theorem 1.1 which shows a possible curvature condition of Riemannian disks containing two min-max capillary embedded geodesics.
\begin{thm}
There are at least 2 capillary embedded geodesics on Riemannian disks $(D^{2},\partial D^{2},g)$ with nonnegative Gaussian curvature in $D^{2}$ and total signed geodesic curvature on boundary satisfies $\int_{\partial D^{2}} \kappa \ge \pi$.
    \begin{proof}
    It suffices to show that if a Riemmanian disk $(D^{2},\partial D^{2},g)$ satisfies the conditions on Gaussian curvature and total signed geodesic curvature then it satisfies the hypothesis $(\star)$. We prove that there is no simple geodesic lasso on $(D^{2},\partial D^{2},g)$.

    Suppose there exists a simple geodesic lasso $\gamma:[0,L] \rightarrow D^{2}$ with a vertex $\gamma(0) = \gamma(L) = p \in \partial D^{2}$. We define two angles $\alpha_{0}, \alpha_{L} \in [0,\pi]$ by angles between tangent vectors at $p$ satisfying
    \begin{align}
         \langle \gamma'(0), v_{p} \rangle &= \cos \alpha_{0}\\
         \langle \gamma'(L), v_{p} \rangle &= \cos \alpha_{L},
    \end{align}
    where $v_{p}$ is a unit tangent vector at $p$ of $\partial D^{2}$. Here we choose the direction of $v_{p}$ to satisfy $\alpha_{0} + \alpha_{L} < \pi$.  Also let us denote the region bounded by $\gamma$ by $\Omega$. Note that $\alpha_{0}$ and $\alpha_{L}$ are exterior angles of $\Omega$.
    
    By Gauss-Bonnet theorem, we know that $\int_{D^{2}} K \le \pi$ from $\int_{\partial D^{2}} \kappa \ge \pi$. Then we have
    \begin{align}
    2 \pi &= \int_{\Omega} K + \int_{\gamma} \kappa + \alpha_{0}+\alpha_{L}\\
    &\le \int_{D^{2}} K + \int_{\gamma} \kappa + \alpha_{0}+\alpha_{L} \\
    &\le \pi + \alpha_{0} + \alpha_{L} < 2 \pi \nonumber
    \end{align}
    and this shows the contradiction. (16) holds since $\Omega$ is topological disk and the angle condition (14) and (15) hold and (17) comes from the nonnegativity of Gaussian curvature on $D^{2}$. Hence there is no simple geodesic lasso on $(D^{2},\partial D^{2},g)$ and this nonexistence implies the hypothesis $(\star)$.
    \end{proof}
\end{thm}

\section{Morse index bound of min-max capillary and free boundary embedded geodesics}
In this section, we prove the generic Morse index bound of min-max capillary embedded geodesics which we constructed in previous sections. Throughout this section, we assume the hypothesis $(\star)$ and the $\theta$-bumpiness of the metric.

By the compactness theorem of capillary varifolds (Corollary A.2) and the bumpy metric theorem of capillary geodesics (Theorem C.1), we have the generic finiteness of capillary embedded geodesics with bounded length. 

We mainly follow the ideas to prove Morse index bound in \cite{K} and explain the necessary modifications in capillary case. We apply curve shortening flow with fixed boundary instead of curve shortening flow in simple closed geodesic cases.
\subsection{Local min-max theorem}
We state the boundary version of local min-max theorem in \cite{K} (See also \cite{MN2} and \cite{W2}). For a domain $\Omega$ and curve $\gamma$ with $\partial_{rel} \Omega = \gamma$, denote $\Omega_{v}$ as a domain induced by $(F_{v})_{\sharp}(\gamma)$ where $\{ F_{v}\}_{v \in \bar{B}^{k}}$ is a local min-max diffeomorphism. By taking the vector fields with normal negative eigensections, the following directly holds by the proof of Theorem 4.1 in \cite{K}:
\begin{thm} [\cite{K}, Theorem 4.1]
Let $(D^{2},\partial D^{2},g)$ be a Riemannian disk with $\theta$-bumpy metric and strictly convex boundary. Let $\gamma$ be an embedded capillary geodesic with contact angle $\theta$ with corresponding domain $\Omega$, Morse index $k$ and multiplicity one. For every $\beta>0$, there is $\epsilon>0$ and a smooth family $\{ F_{v}\}_{v \in \bar{B}^{k}} \subset Diff(D^{2})$ such that
\begin{enumerate}[label=(\roman*)]
\item $F_{0} = Id, \, F_{-v} = F_{v}^{-1}$ for all $v \in \bar{B}^{k}$; 
\item the function
\begin{equation*}
    L^{\theta}_{\gamma}: \bar{B}^{k} \rightarrow [0, \infty), \,\,\,\, L^{\theta}_{\gamma}(v) = L^{\theta}(\Omega_{v}),
\end{equation*}
is strictly concave;
\item $||F_{v}-Id||_{C^{1}}< \beta$ for all $v \in \bar{B}^{k}$;
\end{enumerate}
and such that for every $V \in IV_{1}(D^{2})$ induced by a related boundary of a domain $\Omega'$ whose endpoints are on $\partial D^{2}$ with $F(V, \gamma) < \epsilon$, we have
\begin{equation*}
\max_{v \in \bar{B}^{k}} L^{\theta} (\Omega'_{v})\ge L^{\theta}(\Omega)
\end{equation*}
with equality only if $\gamma = (F_{v})_{\sharp}V$ for some $v \in \bar{B}^{k}$.
\end{thm}
\subsection{Metric perturbation}
We elaborate the metric perturbation in \cite{K} to settle the boundary version of interpolation lemmas. Suppose we have a capillary embedded geodesic $\gamma$ and adopt the Fermi coordinate (12) on the tubular neighborhood $N_{h}(\gamma)$ (on the extended closed surface $M$) in \cite{K}:
 \begin{equation}
ds^{2} = J(x,y)^{2}dx^{2} + dy^{2}
\end{equation}
such that
\begin{enumerate} 
\item $J_{yy} = -KJ$
\item $J(x,0)=1$
\item $J_{x}(x,0) = \kappa =0$
\item $J_{y}(x,0) = 0$
\end{enumerate}
where $K$ is a Gaussian curvature and $\kappa$ is a geodesic curvature of $\gamma$. 

We perturb the metric to be foliated by mean convex curves to obtain a suitable squeezing lemma. We take $M> \max K_{(D^{2},g)}$ and small $h>0$ satisfying $ h<\max(M^{-1},1/10)$. We use the bump function $\phi_{\beta}:D^{2} \rightarrow \mathbb{R}$ constructed in \cite{K}.
\begin{prop} [\cite{K}, Proposition 5.2]
Given a capillary embedded geodesic $\gamma$ with contact angle $\theta$, for any $\beta>0$, there exists a smooth bump function $\phi_{\beta}:D^{2} \rightarrow \mathbb{R}$ such that
\begin{enumerate}[label=(\roman*)]
\item $\phi_{\beta}(x)=0$ when $x \notin N_{h}(\gamma)$
\item $\phi_{\beta}(x)<0$ for $x \in \gamma$
\item $||\exp(2\phi_{\beta}) - 1||_{C^{0}} < \beta$
\item $\nabla_{\dot{\gamma}} \phi_{\beta} =0$ on $\gamma$ [$\phi_{\beta}$ is constant on $\gamma$]
\item The Hessian $\partial^{2} \phi_{\beta} (\nu,\nu)=M$ on $\gamma$.
\end{enumerate}
\end{prop}
For a fixed $\beta >0$, let us consider the metric $g_{\beta}=\exp(2 \phi_{\beta}) g$ whose Fermi coordinate representation is given by  
\begin{equation}
ds_{\beta}^{2} = J(x,y)^{2}\exp(2 \phi_{\beta}) dx^{2} +\exp (2 \phi_{\beta}) dy^{2}
\end{equation}
and we consider a canonical diffeomorphism $G_{\beta}: (D^{2},g) \rightarrow (D^{2},g_{\beta})$ given by $G_{\beta}(z) \equiv z$ for all $z \in D^{2}$. Let $\gamma_{g_{\beta}} := G_{\beta}(\gamma)$ and $N_{h,g_{\beta}}(\gamma_{g_{\beta}}) : = G_{\beta}(N_{h}(\gamma_{g_{\beta}}))$ in $(D^{2},g_{\beta})$. The following is a boundary version of Lemma 5.3 in \cite{K} and it also ensures the conservation of the convexity of the boundary along diffeomorphism $G$.
\begin{lem} 
For any $\beta>0$, $\gamma_{g_{\beta}}$ is a capillary embedded geodesic with contact angle $\theta$. The Gaussian curvature satisfies $K_{g_{\beta}}(z)<0$ for $z \in \gamma_{g_{\beta}}$ on $(D^{2},g_{\beta})$. Moreover, for two endpoints $z_{1}, z_{2}$ of $\gamma_{g_{\beta}}$ on $\partial D^{2}$, the geodesic curvature $\kappa_{\partial D^{2}}$ of $\partial D^{2}$ satisfies $\kappa_{\partial D^{2}}(z_{1}), \kappa_{\partial D^{2}}(z_{2})>0$.
\begin{proof}
Since $G$ is a conformal deformation and by the arguments in Lemma 5.3 in \cite{K}, $\gamma_{g_{\beta}}$ is a capillary embedded geodesic with contact angle $\theta$ and Gaussian curvature is negative on $\gamma_{g_{\beta}}$. Consider $\partial D^{2}$ in the extended surface $M$. By the calculation of the change of the geodesic curvature by conformal deformation in Besse \cite{Be}, we have
\begin{equation}
    \kappa_{g_{\beta}}(z_{i}) = e^{-\phi_{\beta}}\Big(\kappa_{g}(z_{i})-\frac{\partial \phi_{\beta}}{\partial \nu}\Big)
\end{equation}
for $i=1,2$. Note that $\nabla_{\nu}\phi_{\beta}=0$ by the construction of $\phi_{\beta}$ in Proposition 5.2. Hence $\kappa_{g_{\beta}}(z_{i})>0$ for $i=1,2$ by the convexity of the boundary.
\end{proof}
\end{lem}
\subsection{Identifying the endpoints by sliding the curves} We introduce the deformation of the family of domains $\{ \Phi^{j} \}$ to identify all endpoints of relative boundaries to be $\partial \gamma$ with an $F$-distance upper bound along the deformation. We prove this sliding deformation to apply the curve shortening flow with fixed boundary in the procedure to construct squeezing deformation in the next subsection.
\begin{lem}
Let $(D^{2},\partial D^{2},g)$ be a Riemannian disk with $\theta$-bumpy metric and a strictly convex boundary. Let $\gamma = \partial^{i} \Omega$ be a fixed embedded capillary geodesic with contact angle $\theta$ and corresponding region $\Omega$, and $X$ be a finite dimensional simplicial complex. Then there is $\delta_{0}>0$ satisfying the following property: For $0<\delta <\delta_{0}$, if $\Phi:X \rightarrow \mathcal{D}(D^{2})$ satisfying
\begin{equation}
    F(\partial^{i}\Phi(x),\gamma) < \delta,
\end{equation}
then there is a homotopy $H:[0,1] \times X \rightarrow \mathcal{D}(D^{2})$ such that
\begin{equation}
    F(\partial^{i}H(t,x),\gamma) < C(\theta) \sqrt{\delta}
\end{equation}
and two endpoints satisfy $\pi_{i} \partial (H(1,x)) = \pi_{i} \partial( \Omega)$ for any $x \in X$ and $i= 1,2$.
\begin{proof}
    Consider the tubular neighborhood $N_{h}(\gamma)$ for small $h$. Then since $\gamma$ and $\partial M$ intersect with an angle $\theta$, there exists $h_{0}>0$ such that for $0<h<h_{0}$, the Hausdorff distance in $\partial D^{2}$ satisfies
    \begin{equation}
        d_{\mathcal{H}, \partial D^{2}}(\partial \gamma, \partial N_{h}(\gamma) \cap \partial D^{2})< \frac{2h}{\sin \theta}.
    \end{equation}
    Let us set $\delta_{0} = h_{0}^{2}/10$. By (23) and Lemma 3.1 in \cite{K}, $\partial^{i} \Phi(x)$ is supported in $N_{\sqrt{10\delta}}(\gamma) \subset N_{h_{0}}(\gamma)$.
    
    We add two small segments connecting between endpoints of $\partial^{i} \Phi(x)$ and those of $\gamma$ along $\partial D^{2}$ and push this curve slightly to the interior of $D^{2}$, and consider the corresponding domain. Then we connect the family of initial domains $\Phi(x)$ and the new family of domains explained above with small additional segments via homotopy.

Now we consider the exponential map $\exp_{p}: T_{p} \partial D^{2} \rightarrow \partial D^{2}$ on $\partial D^{2}$. For $x \in X$, $i=1,2$ and $t \in [0,1]$, we define two segments $\gamma_{i,t}$ by segments connecting $\pi_{i} \partial \Phi(x)$ and $\exp_{\pi_{i}\partial \Omega} [(1-t)(\exp_{\pi_{i}\partial \Omega})^{-1} \pi_{i} \partial \Phi(x)]$ along $\partial D^{2}$. We define a homotopy $H'(t,x)$ by domains satisfying
\begin{equation}
    \partial^{i} H'(t,x)= \gamma_{1,t} *\partial^{i} \Phi(x) * \gamma_{2,t}
\end{equation}
for $i=1,2$. We regard $\gamma_{i,t}$ as curves in $\mathring{D}^{2}$ here. Then now we prove the $F$-distance estimate for this homotopy. For $f: G_{1}(D^{2}) \rightarrow \mathbb{R}$ satisfying $|f| \le 1$ and $Lip(f) \le 1$ (we regard $f$ on $G_{1}(D^{2} \setminus  int(D^{2}))$ as $f$ on the extended surface $M$ here),
\begin{align}
    \Big| \int f d(\partial^{i}H'(t,x)) - \int f d\gamma \Big| &\le \Big| \int f d(\partial^{i}H'(t,x)) - \int f d(\partial^{i}\Phi(x)) \Big| + \Big| \int f d(\partial^{i}\Phi(x)) - \int f d\gamma \Big| \nonumber \\
    &\le  \Big| \int f d \gamma_{1,t} + \int f d \gamma_{2,t}\Big| + \delta \\
    &\le \nonumber |\gamma_{1,t}| + |\gamma_{2,t}| + \delta \\
    &<  \frac{4 \sqrt{10 \delta}}{\sin \theta} + \delta \\ &< C(\theta) \sqrt{\delta} \nonumber, 
\end{align}
where (25) follows from (21), (24), and (26) comes from (23) and the fact that $\partial^{i}\Phi(x)$ is supported in $N_{\sqrt {10 \delta}} (\gamma)$. Then by the definition of the $F$-distance, we have the following $F$-distance estimate
\begin{equation}
    F(\partial^{i}H'(t,x),\gamma) < C(\theta) \sqrt{\delta}.
\end{equation}

At last, we push the curves along the homotopy $\partial^{i}H'(t,x)$ slightly into the interior of $D^{2}$, and define a corresponding homotopy of domains by $H(t,x)$. By (27), the $F$-distance estimate (22) holds and $H(t,x)$ is a desired homotopy.
\end{proof}
\end{lem}
\subsection{Squeezing lemma} Now we prove the capillary version of squeezing lemma, whose simple closed curve version is Lemma 5.4 in \cite{K}. The main difference with simple closed curve version is the curve shortening flow part. We apply the curve shortening flow with  fixed boundary in Appendix B instead of curve shortening flow for closed curves. We apply Lemma 5.4 to identify the endpoints to squeeze whole families to a single capillary curve, provided a capillary embedded geodesic whose ambient Gaussian curvature is negative on the geodesic. 

The following technical lemma shows the existence of mean convex foliation of $N_{h}(\gamma)$ where fixed points are two endpoints of $\gamma$, in case of a geodesic is strictly stable and achieves a negative ambient Gaussian curvature.  
\begin{lem} There exists $h_{1}>0$ such that for $0<h<h_{1}$, there exists a foliation by curves with fixed endpoints $\{ S_{t} \}_{t \in [0,h]}$ of $N_{h}(\gamma)$ satisfying the following
\begin{enumerate}
    \item $S_{0} = \gamma$,
    \item For $t \in (0,h]$, each leave $S_{t}$ consists of two embedded curves whose endpoints are two endpoints of $\gamma$,
    \item For $t \in (0,h]$, the curvature vector of $S_{t}$ points toward $\gamma$.
\end{enumerate}
\begin{proof}
    We consider $h< \min (inj(D^{2}),h_{0}) $ and set $\partial \gamma = \{ p_{1}, p_{2} \}$, where $h_{0}$ is a constant in the proof of Lemma 5.4 satisfying (23). Note that $\gamma$ has a (rounded) trapezoidal tubular neighborhood $N_{h}(\gamma)$.  For small $h>0$, $\partial N_{h}(\gamma) \cap \partial D^{2}$ consists of four points, where two points are on the neighborhood of $p_{1}$ and two points are on the neighborhood of $p_{2}$. We denote $\partial N_{h}(\gamma) \cap \partial D^{2} = \{ q_{1,1,h},q_{1,2,h}, q_{2,1,h}, q_{2,2,h} \}$, where $q_{1,1,h} , q_{1,2,h} \in N_{2h/ \sin \theta,\partial M}(p_{1})$ and $q_{2,1,h} , q_{2,2,h} \in N_{2h/ \sin \theta,\partial M}(p_{2})$. We label points $q_{i,j,h}$ to be $q_{1,j,h}$ and $q_{2,j,h}$ are on the same side of $\gamma$ for $j=1,2$. Without loss of generality, the angle between a unit tangent vector $v_{i}$ of $\gamma$ at $p_{i}$ toward $int(D^{2})$ and a unit tangent vector $w_{i}$ of $\partial D^{2}$ at $p_{i}$ toward $q_{i,1,h}$ is an acute contact angle $\theta \in (0, \pi/2)$. Then the angle between $v_{i}$ and $-w_{i}$ which is a unit tangent vector of $\partial M$ at $p_{i}$ toward $q_{i,2,h}$ is $\pi - \theta$.  

    For some $\delta>0$, we adopt the Fermi coordinate $c:[-\delta,|\gamma|+ \delta] \times (-h,h) \rightarrow M$ on $N_{h}(\gamma)$  in (18), where $c(0,0) = p_{1}$ and $c(|\gamma|,0) = p_{2}$.  We set $\delta>0$ here to satisfy $N_{h}(\gamma) \subset c([-\delta,|\gamma|+ \delta] \times (-h,h))$. Also we define $\pi_{y}: [-\delta,|\gamma|+ \delta] \times (-h,h) \rightarrow (-h,h)$ by $y$-coordinate function of Fermi coordinate $c$. Without loss of generality, we take Fermi coordinate $c$ to satisfy $\pi_{y} c^{-1} q_{i,1,h}>0$ and $\pi_{y} c^{-1} q_{i,2,h} <0$ for $i=1,2$.

    We divide $N_{h}(\gamma)$ into four disjoint small convex regions $\{ R_{i,j,h} \}_{i,j \in \{1,2 \}}$ where each $R_{i,j,h}$ contains $q_{i,j,h}$, respectively, and the interior region $N_{h}(\gamma) \setminus \cup_{i,j=1,2} R_{i,j,h}$. We construct the foliations on $N_{h}(\gamma) \setminus \cup_{i,j=1,2} R_{i,j,h}$ and $R_{i,j,h}$'s separately and join those to a combined foliation.

    For $i=1,2$, let us connect $p_{i}$ and $q_{i,1,h}$ by geodesics $\alpha_{i,1,h}$. We define the small regions enclosed by $\partial D^{2}$ and $\alpha_{i,1,h}$ by $R_{i,1,h}$. Also we define two vertical geodesics $\alpha_{1,2,h}:= c(\{ 0 \} \times [-h,0])$ and $\alpha_{2,2,h} : = c(\{ |\gamma| \} \times [-h,0])$. We denote the small regions enclosed by $\partial D^{2}$, $\alpha_{i,2,h}$ and $\partial N_{h}(\gamma)$ by $R_{i,2,h}$.

    Let us construct the foliation in the interior part first.  As in \cite{K}, for $t \in [-h,h]$ we take
    \begin{equation}
    \gamma_{int,t,h}:= c( \{y=t \}) \cap (N_{h} \setminus \cup_{i,j=1,2} R_{i,j,h}).
    \end{equation}
    There exists $h'>0$ such that for $0<h<h'$, by the arguments in Section 5.3 of \cite{K}, the curvature vector of each leave $\{ \gamma_{int,t,h} \}_{t \in [-h,h]}$ points toward into $\gamma$.

    Now we consider each small boundary region $R_{i,j,h}$ for $i,j \in \{ 1,2 \}$. Note that $R_{i,j,h}$ is a convex region so that is star-shaped. We foliate $R_{i,j,h}$ by scaled curves of parts of $\partial R_{i,j,h} \setminus \alpha_{i,j,h}$. For $s \in [0,1]$, let us define
    \begin{equation}
        \gamma_{i,j, s,h} = \exp _{p_{i}}(s(\exp_{p_{i}})^{-1}(\partial R_{i,j,h} \setminus \alpha_{i,j,h})),
    \end{equation}
    which foliates $R_{i,j,h}$ with a fixed point $p_{i}$ by star-shapedness of $R_{i,j,h}$ and $h< inj(D^{2})$. Since $\partial \tilde{B}_{r}(p_{i})$ is convex for small $r>0$, $\partial D^{2}$ is convex, and each $\gamma_{i,j,s,h}$ consists of the scaled curve of these based on $p_{i}$, there exists $h''>0$ such that each leave $\gamma_{i,j,s,h}$ has a curvature vector points toward $\gamma$ for $0<h<h''$. 
    
    Now we connect the foliations we constructed in (28) and (29). Let us parametrize $\alpha_{i,1,h}$ by the regular curve by $\alpha_{i,1,h}: [0, 1] \rightarrow D^{2}$ where $\alpha_{i,1,h}(t) = \exp _{p_{i}}(t(\exp_{p_{i}})^{-1}(q_{i,1,h}))$ for $i=1,2$.  Since $\gamma$ and $\partial D^{2}$ meets with capillary angle $\theta$ transversely and $\partial D^{2}$ is convex, there exists $h'''>0$ such that for $0<h<h'''$, $y_{\alpha_{i,1,h}}:=\pi_{y} c^{-1} \alpha_{i,1,h}: [0, 1] \rightarrow [0,h]$ is a strictly increasing function. We define the foliation $\{ \gamma_{t,h} \}_{t \in [0,h]}$ by
    \begin{equation}
        \gamma_{t,h} := (\gamma_{1,1,y_{\alpha_{1,1,h}}^{-1}(t),h}* \gamma_{int,t,h}*\gamma_{2,1,y_{\alpha_{2,1,h}}^{-1}(t),h}) \cup (\gamma_{1,2,t/h,h}* \gamma_{int,-t,h}*\gamma_{2,2,t/h,h}).
    \end{equation}
    Then we take $h_{1} = \min \{inj(D^{2}),h_{0},h',h'',h''' \}$. By taking $S_{t} = \gamma_{t,h}$ in (30), by constructions in (28) and (29), the foliation by curves with fixed endpoints $\{ S_{t} \}_{t \in [0,h]}$ satisfies conditions (1), (2) and (3) for $0<h<h_{1}$.
\end{proof}
\end{lem}
Let us call $N_{h}(\gamma)$ having a mean convex foliation with fixed point $\{ S_{t} \}_{t \in [0,h]}$ as a \emph{mean convex neighborhood of $\gamma$}. Following is a squeezing lemma which is a capillary analog of Lemma 5.4 in \cite{K}.

\begin{lem}
Let $(D^{2},\partial D^{2},g)$ be a Riemannian disk with a $\theta$-bumpy metric and strictly convex boundary, $\gamma$ be a capillary embedded geodesic, $\Omega$ be a corresponding domain with $\gamma = \partial_{rel} \Omega$, and $X$ be a simplicial complex with finite dimension $k$. There exists $\delta_{1} = \delta_{1}((M,g), \gamma)>0$ with the following property:

For $0<\delta<\delta_{1}$, if $\Phi : X \rightarrow \mathcal{D}(D^{2})$ is a continuous map in the smooth topology such that
\begin{equation*}
    \sup \{ F(\partial^{i} \Phi(x), \gamma) : x \in X \} < \delta,
\end{equation*}
then there is a homotopy $H:[0,1] \times X \rightarrow \mathcal{D}(D^{2})$ such that $H(0,x) = \Phi(x)$ and $H(1,x) = \Omega$ so that $\Phi$ is nullhomotopic.
\begin{proof}
We mainly discuss the difference with the proof of Lemma 5.4 in \cite{K}. We consider a capillary embedded geodesic whose ambient negative Gaussian curvature is negative case first. We can assume that the only capillary embedded geodesic in the tubular neighborhood $N_{h}(\gamma)$ is $\gamma$ since the metric is $\theta$-bumpy.

We take $\delta_{1} = h_{1}^{2}/10$ where $h_{1}$ is a constant in Lemma 5.5. Note that $0<\delta_{1}<\delta_{0}$ by the construction of mean convex neighborhood in the proof of Lemma 5.5. We apply Lemma 5.4 for $0< \delta <\delta_{1}$ and obtain a homotopy $H_{1}:[0,1] \times X \rightarrow \mathcal{D}(D^{2})$ such that
\begin{equation*}
    F(\partial^{i}H_{1}(t,x),\gamma) < C(\theta) \sqrt{\delta}
\end{equation*}
and endpoints are fixed as $\pi_{i} \partial H_{1}(1,x) = \pi_{i} \partial \Omega$ for $i=1,2$. The entire support of $\partial^{i}H_{1}(1,x)$ are in $N_{\sqrt{10\delta}}(\gamma)$ by Lemma 3.1 in \cite{K}. 

Now we apply the curve shortening flow with boundary (see Appendix B) to curves $\{ \partial^{i}H_{1}(1,x) \}_{x \in X}$. $N_{h_{0}}(\gamma)$ is a mean convex neighborhood by Lemma 5.5. We define $Y(\cdot,t)|_{\partial^{i}H_{1}(1,x)}$ to be the curve shortening flow with fixed boundary, where $Y(\cdot,0)|_{\partial^{i}H_{1}(1,x)}=\partial^{i}H_{1}(1,x)$ and $Y_{t}(y,t)|_{\partial^{i}H_{1}(1,x)} = \kappa N(y)$ where $\kappa$ is a geodesic curvature and $N$ is a normal vector field on $Y_{t}(\cdot,t)|_{\partial^{i}H_{1}(1,x)}$ for $t>0$. Then by the regularity theorem of curve shortening flow with boundary (Theorem B.1) and mean-convexity of $N_{h_{0}}(\gamma)$, for any small $s>0$ there is a large time $t_{0}$ such that for $t>t_{0}$ and any $x \in X$, the total curvature $\int_{Y(\cdot,t)} |\kappa| < s$ since curves uniformly smoothly converge to $\gamma$ (See also Theorem 2.6 of \cite{H}). Moreover, the curves are graphical over $\gamma$. Now we construct the homotopy $H_{2} : [0,1] \times X \rightarrow \mathcal{D}(D^{2})$ induced by the flow by $H_{2}(t,x)$ satisfying
\begin{equation*}
    \partial^{i} H_{2}(t,x) = Y(\cdot,2t_{0}t)|_{\partial^{i}H_{1}(1,x)},
\end{equation*}
for all $x \in X$ and $t \in [0,1]$. We apply the squeezing homotopy as in \cite{K} and define this homotopy as $H_{3}:[0,1] \times X \rightarrow \mathcal{D}(D^{2})$. By combining three homotopies, we obtain a desired homotopy $H:[0,1] \times X \rightarrow \mathcal{D}(D^{2})$ as
\begin{equation}
    H = H_{1} \cdot H_{2} \cdot H_{3}.
\end{equation}

Now we consider the general cases. We fix some $\beta>0$ and consider the perturbed metric $(S^{2},g_{\beta})$ in (19). By Lemma 5.3 and Lemma 5.5, there exists a mean convex neighborhood $N_{h'}(\gamma_{g_{\beta}})$ where $N_{h'}(\gamma_{g_{\beta}}) \cap \partial M$ is strictly convex. Let us take $\delta_{1} = h'^{2}/20$. We complete the proof by following remaining arguments in the proof of Lemma 5.4 in \cite{K} and taking the pullback map of (31).
\end{proof}
\end{lem}
\subsection{$F$-distance and $L^{\theta}$-functional estimate} In this section, we prove the $F$-distance and $L^{\theta}$-functional estimate along the homotopy we constructed in Lemma 5.6. This is an analog of Section 6 of \cite{K}. We will focus on the modification we need in the capillary case.

Let $\gamma$ be a capillary embedded geodesic with a contact angle $\theta \in (0, \pi/2)$ in $(D^{2},\partial D^{2}, g)$ and Gaussian curvature on $\gamma$ is negative. Take $h_{1}>0$ in Lemma 5.5. Suppose $0<h<h_{1}$ in the following Lemma 5.7 and Lemma 5.8, then $N_{h}(\gamma)$ is mean convex neighborhood by Lemma 5.5. We take the Fermi coordinate $c: [-\delta,|\gamma|+\delta] \times (-h,h) \rightarrow M$ of (19). The following lemma is a capillary analog of Lemma 6.1 in \cite{K}. The proof of the following lemma directly follows from the proof of Lemma 6.1 in \cite{K}.
\begin{lem}
    Let $(D^{2},\partial D^{2},g)$, $\gamma$, $c$ be as above and $K(z)<0$ for $z \in N_{h}(\gamma)$. There exists $C=C(|\gamma|)>0$ satisfying the following property: For $0<\epsilon<h^{2}$, if an embedded curve $\alpha$ has its endpoints on different components of $\partial D^{2} \cap N_{h}(\gamma)$ satisfies $|\alpha|<|\gamma|+\epsilon$, then
\begin{equation}
F(\alpha,\gamma)< C(|\gamma|)(h+ \sqrt{\epsilon}).    
\end{equation}
\end{lem}

Then we obtain the $F$-distance estimate along the squeezing map in Lemma 5.6 in negative ambient Gaussian curvature case. This is an analog of Lemma 6.2 in \cite{K}. The sliding procedure in Lemma 5.4 causes the change of order on the $F$-distance bound.
\begin{lem} Let $\gamma= \partial \Omega$ be a capillary embedded geodesic with contact angle $\theta$ on $(D^{2},\partial D^{2},g)$, Gaussian curvature $K(z)<0$ for $z \in N_{h}(\gamma)$, and $X$ be a $k$-dimensional simplicial complex. There exists $\epsilon_{0} = \epsilon_{0}(h,\theta)$ such that  there exists $C = C(|\gamma|,\theta)>0$ satisfying the following property: For $0<\epsilon<\epsilon_{0}$, if $\Phi:X \rightarrow \mathcal{D}(D^{2})$ is a continuous map in the smooth topology such that
\begin{equation*}
    \sup \{ F(\partial^{i} \Phi(x), \gamma) : x \in X \} < \epsilon,
\end{equation*}
then there is a homotopy $H:[0,1] \times X \rightarrow \mathcal{D}(D^{2})$ such that $H(0,x) = \Phi(x)$, $H(1,x) = \Omega$ and the following $F$-distance estimate holds along $H$:
\begin{equation*}
    \sup \{ F(\partial^{i}H(t,x),\gamma) : x \in X \text{ and } t \in [0,1] \} < C(|\gamma|,\theta) \sqrt[4]{\epsilon}.
\end{equation*}
\begin{proof}
    Take $\epsilon_{0} = \min \{1, C(\theta)^{-2}h^{2}/10 \} h^{2}/10 $ where $C(\theta)$ is the constant in (22). We consider the homotopy $ H = H_{1} \cdot H_{2} \cdot H_{3}$ in (31). We apply Lemma 5.4 and obtain the following $F$-distance estimate:
    \begin{equation}
        \sup \{ F(\partial^{i}H_{1}(t,x),\gamma) : x \in X \text{ and } t \in [0,1] \} < C(\theta) \sqrt{\epsilon}.
    \end{equation}
    Note that $C(\theta) \sqrt{\epsilon}< h^{2}/10$ and $supp(\partial^{i}H_{1}(t,x)) \subset N_{\sqrt{10\epsilon}}(\gamma)$ for $t \in [0,1]$ and $x \in X$. Since $0<\sqrt{10\epsilon}<h$, $N_{\sqrt{10\epsilon}}(\gamma)$ is a mean convex neighborhood of $\gamma$ by Lemma 5.5.

    Since the length functional monotonically decreases by the curve shortening flow with fixed boundary, the arguments in Lemma 6.2 in \cite{K} can be directly applied to the homotopy $H_{2} \cdot H_{3}$ by applying Lemma 5.7 in place of Lemma 6.1 in \cite{K}. We have following for $t \in [0,1]$ and $x \in X$:
    \begin{align}
        F(\partial^{i}H(t,x),\gamma) &< C(|\gamma|) (\sqrt{10\epsilon}+(\sup \{ F(\partial^{i}H_{1}(1,x),\gamma) : x \in X\})^{\frac{1}{2}}) \\
        &< C(|\gamma|)(\sqrt{10\epsilon}+ C(\theta)^{\frac{1}{2}}\sqrt[4]{\epsilon}) \\
        &< 2C(|\gamma|)C(\theta)^{\frac{1}{2}} \sqrt[4]{\epsilon}.
    \end{align}
    (34) follows from (32) in Lemma 5.7 and (35) is obtained by (33). By taking $C(|\gamma|,\theta) = 2C(|\gamma|)C(\theta)^{\frac{1}{2}}$ in (36), we obtain a desired estimate.
\end{proof}
\end{lem}
Then with Lemma 5.3 and Lemma 5.6, we can apply the arguments in Lemma 6.3 in \cite{K}, and we obtain the following $F$-distance estimate in general cases.
\begin{lem} Let $\gamma=\partial \Omega$ be a capillary embedded geodesic on $(D^{2},\partial D^{2},g)$, and $X$ be a $k$-dimensional simplicial complex. There exists $C = C(|\gamma|)>0$ and $\epsilon_{0}= \epsilon_{0}(\gamma)>0$ such that satisfying the following property: For $0<\epsilon<\epsilon_{0}$, if $\Phi : X \rightarrow \mathcal{D}(D^{2})$ is a continuous map in the smooth topology such that
\begin{equation*}
    \sup \{ F(\partial^{i}\Phi(x), \gamma) : x \in X \} < \epsilon,
\end{equation*}
then there is a homotopy $H:[0,1] \times X \rightarrow \mathcal{D}(D^{2})$ such that $H(0,x) = \Phi(x)$, $H(1,x) = \Omega$ and the following $F$-distance estimate holds along $H$:
\begin{equation*}
    \sup \{ F(\partial^{i}H(t,x),\gamma) : x \in X \text{ and } t \in [0,1] \} < C(|\gamma|,\theta) \sqrt[4]{\epsilon}.
\end{equation*}
\end{lem}

The following interpolation lemma associated with a local min-max theorem (Theorem 5.1) is an analog of Lemma 6.4 in \cite{K}.

\begin{lem} Let $\gamma=\partial \Omega$ be a capillary embedded geodesic with contact angle $\theta$ and $X$ be a simplicial complex with finite dimension $k$. Suppose $\Phi: X \rightarrow \mathcal{D}(D^{2})$ be a $k$-parameter family of domains and $\{ F_{v}\}_{v \in \overline{B}^{j}}$ be a local min-max family of diffeomorphisms in Theorem 5.1. There exists $\epsilon_{1}=\epsilon_{1}(\gamma)>0$ such that satisfying the following property : If there is a continuous function $w: X \rightarrow \partial B^{j}$ satisfying
\begin{equation*}
F(\partial^{i}\Phi(x),(F_{w(x)})_{\sharp}\gamma)<\epsilon
\end{equation*}
with $0<\epsilon<\epsilon_{1}$, then there exists a homotopy $H: [0,1] \times X \rightarrow \mathcal{D}(D^{2})$ such that $H(0,x) = \Phi(x)$, $H(1,x) = (F_{w(x)})_{\sharp}\Omega$ satisfying
\begin{equation}
    \sup \{ F(\partial^{i}H(t,x),(F_{w(x)})_{\sharp}\gamma) : x \in X \text{ and } t \in [0,1] \} < C(|\gamma|,\theta) \sqrt[4]{\epsilon}.
\end{equation}
\end{lem}

Finally we obtain the following $L^{\theta}$-functional estimate along the squeezing homotopy.
\begin{cor}
Suppose $\gamma$, $\Omega$, $\Phi$, $\{ F_{v}\}_{v \in \overline{B}^{j}}$ are as in Lemma 5.10, and $\epsilon_{1}=\epsilon_{1}(\gamma)$ is obtained by Lemma 5.10. Then for $0<\epsilon<\epsilon_{1}$, the following holds: If there is a continuous function $w: X \rightarrow \partial B^{j}$ satisfying
\begin{equation*}
F(\partial^{i}\Phi(x),(F_{w(x)})_{\sharp}\gamma)<\epsilon
\end{equation*}
with $0<\epsilon<\epsilon_{1}$, then there exists a homotopy $H: [0,1] \times X \rightarrow \mathcal{D}(D^{2})$ such that $H(0,x) = \Phi(x)$, $H(1,x) = (F_{w(x)})_{\sharp}\Omega$ satisfying
\begin{equation}
    L^{\theta}(H(t,x)) < L^{\theta}((F_{w(x)})_{\sharp}\Omega)+ C(|\gamma|,\theta) \sqrt[4]{\epsilon}
\end{equation}
for $x \in X$ and $t \in [0,1]$.
\begin{proof}
    Let us recall the definition (1) of $L^{\theta}$-functional. Note that the homotopy satisfying (37) in Lemma 5.10 is constructed by the pullback map of (31) in Lemma 5.6. Moreover, by tracking the homotopies consisting (31) of, we have the pullback map of $H_{1}$ is the only homotopy involves the boundary part of $L^{\theta}$-functional.

    By (iii) of Theorem 5.1, let us take local min-max ball to satisfy
    \begin{equation}
        |(F_{v})_{\sharp}(\alpha)| < 2|\alpha|
    \end{equation}
    for $v \in \overline B^{j}$ and $\alpha \in V_{1}(M)$. Moreover, by (43) and the choice of $\beta$ in the proof of Lemma 6.3 in \cite{K}, we take $G_{\beta}$ such that
    \begin{equation}
        |(G_{\beta})^{-1}_{\sharp}(\alpha)| < 3|\alpha|
    \end{equation}
    for $\alpha \in V_{1}(M)$. Then we obtain
    \begin{align}
        L^{\theta}(H(t,x)) &= |\partial^{i}H(t,x)| + \cos \theta |\partial^{b}H(t,x)| \nonumber \\
        &< |(F_{w(x)})_{\sharp}\gamma| + \cos \theta |\partial^{b}H(t,x)| + C(|\gamma|,\theta) \sqrt[4]{\epsilon} \\
        &\le |(F_{w(x)})_{\sharp}\gamma| + \cos \theta (|\partial^{b} (F_{w(x)})_{\sharp}\Omega|+|\partial^{b} (F_{w(x)})_{\sharp}\Omega \triangle \partial^{b}H(t,x)|) + C(|\gamma|,\theta) \sqrt[4]{\epsilon} \nonumber \\ 
        &\le |(F_{w(x)})_{\sharp}\gamma| + \cos \theta |\partial^{b} (F_{w(x)})_{\sharp}\Omega| + 6 \cos \theta \cdot \frac{4 \sqrt{10 \epsilon}}{\sin \theta} + C(|\gamma|,\theta) \sqrt[4]{\epsilon}  \\ \nonumber
         &< L^{\theta}((F_{w(x)})_{\sharp}\Omega)+ C(|\gamma|,\theta) \sqrt[4]{\epsilon}.
    \end{align}
We obtain (41) by applying (37), and (42) by applying (26), (39), and (40).
\end{proof}
\end{cor}
\subsection{Morse Index bound}
Now we prove the Morse index bound of two capillary embedded geodesics on Riemannian disk with a convex boundary $(D^{2},\partial D^{2})$ endowed with a $\theta$-bumpy metric $g$ and the hypothesis $(\star)$. $\theta$-Bumpy metrics are generic in Baire sense by Theorem C.1.
\begin{thm}
Suppose $(D^{2},\partial D^{2}, g)$ is a Riemannian $2$-disk whose boundary is strictly convex and satisfy the hypothesis $(\star)$ with a $\theta$-bumpy metric. Then for each $k=1,2$, there exists a capillary embedded geodesic $\gamma_{k}= \partial_{rel} \Omega_{k}$ with
\begin{equation*}
    index(\gamma_{k}) = k
\end{equation*}
and the $L^{\theta}$-functional of two geodesics satisfy $L^{\theta}(\Omega_{1})<L^{\theta}(\Omega_{2})$.
\begin{proof}
    We explain the necessary modifications from Theorem 7.1 in \cite{K}. We apply Corollary A.3 for the finiteness of capillary embedded geodesics with bounded $L^{\theta}$-functional, and Theorem 5.1 for the capillary version of local min-max theorem.

    Then for each capillary geodesic $\gamma = \partial_{rel}\Omega$ with an index $j<k$, we take the local min-max ball $\{ F_{v} \}_{v \in \overline{B}^{j}}$ to satisfy
    \begin{equation}
    F((F_{v})_{\sharp}(V_{1}),(F_{v})_{\sharp}(V_{2})) \le 2 F(V_{1},V_{2})
\end{equation}
for $V_{1}, V_{2} \in V_{1}(M)$. If $j=1$ and $k=2$, then there exists some $b(\gamma)$ such that
\begin{equation}
    L^{\theta}((F_{v})_{\sharp}(\Omega)) < L^{\theta}(\Omega) - b(\gamma),
\end{equation}
where $\partial_{rel}(\Omega) = \gamma$ and for $v \in \partial B^{1}$ by the concaveness of $L^{\theta}$-functional in the local min-max ball. We pick small $0<s<\min_{\Omega \in D_{w_{k}}} \min(\epsilon(\gamma), b(\gamma)^{4} C(L_{k},\theta)^{-4}/3^{5} ,\epsilon_{1}(\gamma)/3)$, where $\epsilon(\gamma)$ is a constant in Theorem 5.1, $b(\gamma)$ is a constant in (44), $\epsilon_{1}(\gamma)$ is a constant in Corollary 5.11 for $\gamma$ and $C(w_{k},\theta)$ is a constant in (38). By tracking the proof of Lemma 6.1 in \cite{K} and Lemma 5.8, we notice that $C_{\theta}(t) := C(t,\theta)$ can be chosen to be a monotone increasing function. We take $C(w_{k}, \theta)$ here since $w_{k} \ge |\gamma|$ for every capillary embedded geodesic $\gamma = \partial_{rel}\Omega$ with $L^{\theta}(\Omega) = w_{k}$ by the definition of $L^{\theta}$-functional.

We apply Proposition 2.2 for the compact set $\mathcal{S} = \mathcal{D}_{\theta,w_{k}}$ which is compact in the varifold $F$-topology of its corresponding capillary varifolds: There exists $\delta_{s}$ such that
\begin{equation*}
F(\partial ^{\theta} \Omega',\partial ^{\theta} \Omega) < \delta_{s} \Rightarrow \,\, F(\partial ^{i} \Omega',\partial ^{i} \Omega) < s.
\end{equation*}

We apply Theorem 3.6 to obtain a tightened sequence $\{\hat{\Phi}^{j}(x)\}_{x \in dom(\Phi)}$ with the Morse index bound:
\begin{equation}
        \{ \partial^{\theta}\hat{\Phi}^{j} \in V_{1}(D^{2}) : L^{\theta}(\hat{\Phi}^{j}) \ge w_{k}-a \} \subset \bigcup_{\gamma \in \Lambda_{i,w_{k},k}} B^{F}_{\delta_{s}}(\gamma).
        \end{equation}
Then $\Lambda(\{ \hat{\Phi}_{i}\}) = \{ \tilde{\Omega}_{1}, ..., \tilde{\Omega}_{q} \}$. Denote corresponding capillary geodesics $\tilde{\gamma}_{j}$ to be $\tilde{\gamma}_{j}=\partial^{i} \tilde{\Omega}_{j}$. We proceed by the contradiction argument. Let us assume that $index(\tilde{\gamma}_{j})<k$ for every $i=1,...,q$.

Consider the $m_{i}$'th barycentric subdivision $Y_{i}$ of $Dom(X)$ so that $|L^{\theta}(\hat{\Phi}_{i}(x))-L^{\theta}(\hat{\Phi}_{i}(y))|<a/2$ whenever $x,y$ is in a same simplex in $Dom(X)$. Note that $Dom(X) = I$ for $k=1$ and $Dom(X) = S^{1} \times I$ for $k=2$. We define $W_{i}$ to be a union of all $k$-dimensional simplices $t \in Y_{i}$ such that $L^{\theta}(\hat{\Phi}_{i}(x)) \ge w_{k}- a/2$ for some $x \in t$. Then $L^{\theta}(\hat{\Phi}_{i}(y)) \ge w_{k}- a$ for every $y \in W_{i}$. Let connected components be $W_{i,1}, ...,W_{i,r}$. Then by (45), for each $1 \le p \le r$ there exists $1 \le q_{p} \le q$ with
\begin{equation*}
    F(\partial^{i} \hat{\Phi}_{i}(y), \tilde{\gamma}_{q_p}) <\delta_{s}.
\end{equation*}
Note that 
\begin{equation*}
    w_{k}-a \le L^{\theta}(\hat{\Phi}_{i}(y)) \le w_{k}-a/2
\end{equation*}
since $y \in W_{i,p}$ since it belongs to both $W_{i}$ and $Y_{i} \setminus W_{i}$.

Let us assume $j=0$. Then arguments in Theorem 7.1 in \cite{K} apply and conclude to $\partial W_{i,p} =0$. However, since $Dom(X)$ is connected, this cannot happen and $j \ge 1$.

Now we consider the case of $k=2$ and $j=1$. The boundary of each $W_{i,p}$ can be separated to simple closed curves. Let us take an outermost curve $\alpha_{i,p}$ among this curves hence $supp(\alpha_{i,p}) \subseteq \partial W_{i,p}$. Let us define a disk $D_{i,p} \in S^{1} \times I$ to be a disk separated by $\alpha_{i,p}$ containing $W_{i,p}$. Moreover, for some $\{p_{1}, p_{2}\} \subseteq \{1, ..., q\}$, if some $W_{i,p_{1}}$ is contained in the region $D_{i,p_{2}}$, we do not perform the deformation for $\alpha_{i,p_{1}}$. From now, let us only consider curves $\alpha_{i,p}$ which we will apply the deformations, and renumber these curves to $\alpha_{i,1}, ..., \alpha_{i,p'}$. We apply the rest of deformation arguments of Theorem 7.1 in \cite{K} to curves $\{ \alpha_{i,l} \}_{l=1}^{p'}$. Note that interpolation homotopy $H_{2}$ contracts family of domains to a single domain ($(F_{1})_{\sharp} \Omega_{q_p}$ or $(F_{-1})_{\sharp} \Omega_{q_p}$) if $j=1$. Then now consider the family of domains $\Psi_{i}$ defined by
\begin{equation*}
    \Psi_{i} = \hat{\Phi}_{i}(S^{1} \times I \setminus \cup_{j=1}^{p'}D_{i,j})\cup_{\cup_{j=1}^{p'}\alpha_{i,j}} (H_{1} \cdot H_{2})_{\sharp} ([0,1] \times\cup_{j=1}^{p'}\alpha_{i,j}).
\end{equation*}
Notice that $\Psi_{i}$ is a family of domains constructed by the replacement of $\hat{\Phi}_{i}(D_{i,j})$'s by the family of domains $(H_{1} \cdot H_{2})_{\sharp} ([0,1] \times\alpha_{i,j})$. Hence we can reparametrize $\Psi_{i}$ to be a $2$-sweepout $\hat{\Psi}_{i}$ for large $i$ and this satisfies $\sup_{x \in S^{1} \times I}L^{\theta}(x) < w_{k} - \min (b/6,a/2)$ and this gives the contradiction.
\end{proof}

\end{thm}
Also we obtain a general Morse index characterization by bumpy approximation as in Corollary 7.3 in \cite{K}:
\begin{cor}
For a $2$-Riemannian disk $(D^{2},\partial D^{2},g)$ with a strictly convex boundary satisfying the hypothesis $(\star)$, for $k=1,2$, there exists a capillary embedded geodesic $\gamma_{k}$ with
\begin{equation*}
    index(\gamma_{k}) \le k \le index(\gamma_{k})+ nullity (\gamma_{k}).
\end{equation*}
\end{cor}

\section{The optimality of the condition}

In this section, we prove that the boundary total curvature condition in Theorem 4.4 is sharp. Our proof of the following Proposition 6.1 relies on the elementary construction of ``a rounded Euclidean cone". 

\begin{prop} For a given $k \in (0,\pi)$, there exists a Riemannian $2$-disk $(D^{2},\partial D^{2},g)$ with a strictly convex boundary, interior nonnegative Gaussian curvature and a total signed curvature $\int_{\partial D^{2}} \kappa = k$ such that there exists $\theta \in (0, \pi/2)$ satisfying the following: The Riemannian disk $(D^{2},\partial D^{2},g)$ does not admit a capillary embedded geodesic with a contact angle $\theta$.
\begin{proof}
    Consider a Riemannian disk $M$ of revolution explicitly given by $(u,r(u) \cos t, r(u) \sin t)$ for $u \in [s,1]$ and $t \in [0,2 \pi]$ in $\mathbb{R}^{3}$ endowed with Euclidean metric such that
    \begin{enumerate}
        \item $0<s < 1/2 \cos(k/2)$ and $r(s)=0$,
        \item $r(u) \ge 0$ for $u \in [s,1]$,
        \item $r'_{+}(s)= + \infty$ and $r'(u)>0$,
        \item $r''(u) \le 0$,
        \item $r(u)=(k/\sqrt{4 \pi^{2}-k^{2}})u$ for $u \ge 1/2 \cos (k/2)$.
    \end{enumerate}
    By a simple calculation, we can obtain total signed curvature $\int_{\partial D^{2}} \kappa = k$. Moreover, the condition (4) gives $K \ge 0$ on $M$. Then since a flat plane and a cone are isometric, by straightforward trigonometry, we obtain that a geodesic emanating from a boundary point $(1,k/\sqrt{4 \pi^{2}-k^{2}},0)$ with a contact angle $k/2$ is a critical lasso. Note that the geodesic only passes through the cone part ($x \ge 1/2 \cos (k/2)$). Notice that a small addition of the contact angle of a critical lasso on an Euclidean cone induces an interior self-intersection point. Hence there exists small $\epsilon>0$ such that the geodesic emanating from $(1,k/\sqrt{4 \pi^{2}-k^{2}},0)$ with a contact angle $k/2 + \epsilon \in (0,\pi/2)$ which develops an interior self-intersection point. Since $M$ is rotationally symmetric, we see that $M$ does not admit a capillary embedded geodesic with a contact angle $k/2 + \epsilon$. 
\end{proof}
\end{prop}

 \appendix \section{Compactness theorem of embedded geodesics}
The following theorem asserts the compactness of the set of embedded geodesics on surfaces.
\begin{thm} Let $(M,\partial M,g)$ be a closed surface with strictly convex boundary. For a fixed contact angle $\theta$, if $\{ \gamma_{k} \}$ is a sequence of connected embedded geodesics whose endpoints are on $\partial M$ with length bound
\begin{equation*}
    \mathcal{H}^{1}(\gamma_{k}) \le L < \infty
\end{equation*}
for some fixed constant $L \in \mathbb{R}$ independent of $k$. Then up to subsequence, there exists a connected capillary embedded geodesic $\gamma$ with a contact angle $\theta$ where $\gamma_{k} \rightarrow \gamma$ in the varifold sense with
\begin{equation*}
    \mathcal{H}^{1}(\gamma) \le L < \infty
\end{equation*}
and convergence is smooth and graphical for all $x \in M$. Moreover, if for every $k$, $\gamma_{k}$'s are capillary embedded geodesics with a contact angle $\theta$, then $\gamma$ is also a capillary embedded geodesic with a contact angle $\theta$. The multiplicity of convergence $m$ is 1. Furthermore, if $\gamma_{k} \cap \gamma = \emptyset$ eventually, then $\gamma$ is stable, and $index (\gamma) \ge 1$ otherwise.
\begin{proof}
We can directly adopt the compactness theorem of embedded geodesics without boundary of Theorem A.1 in \cite{K}. For the multiplicity part, a small tubular neighborhood of capillary embedded geodesic $\gamma$ is diffeomorphic to rectangular strips and this gives multiplicity $1$. Moreover, by standard arguments of elliptic equations, we can ensure the fact that the first eigenfunction $f$ associated with the second variation formula (1) is a signed function and this gives a stability analysis as in Claim 4 of Theorem A.1 in \cite{K}. The contact angle condition is preserved along the convergence. 
\end{proof}
\end{thm}
Following corollary shows the compactness property of associated capillary varifolds of capillary embedded geodesics.
\begin{cor} Let $(M,\partial M,g)$ be a closed surface with strictly convex boundary. Suppose that a sequence of capillary geodesics $\{ \gamma_{k} \}$ with contact angle $\theta$ satisfies $\partial^{i} \Omega_{k} = \gamma_{k}$ where $\Omega_{k} \in \mathcal{D}(M)$ and the $L^{\theta}$-functional bound as
\begin{equation*}
L^{\theta}(\Omega_{k}) \le L < \infty.
\end{equation*}
Then up to subsequence, there exists a domain $\Omega \in \mathcal{D}(M)$ such that $\partial^{i}\Omega = \gamma$ and an associated capillary varifold $\partial^{\theta}\Omega$ with
\begin{equation*}
    L^{\theta}(\Omega) \le L < \infty.
\end{equation*}
Hence, we have a varifold convergence of capillary varifolds $\partial^{\theta}\Omega_{k} \rightarrow \partial^{\theta} \Omega$.
\begin{proof}
By Theorem A.1 above, we have a smooth convergence $\gamma_{k} \rightarrow \gamma$ with the conservation of contact angles. This gives a convergence of endpoints of $\gamma_{k}$ to endpoints of $\gamma$, and convexity of boundary implies the existence of domain $\Omega \in \mathcal{M}$ such that $\gamma = \partial \Omega$ and associated capillary varifolds $\partial^{\theta}\Omega$. Then we obtain the convergence of varifolds $\partial^{i} \Omega_{k} \rightarrow \partial^{i} \Omega$ and $\partial^{b} \Omega_{k} \rightarrow \partial^{b} \Omega$, and convergence of capilary varifolds $\partial^{\theta} \Omega_{k} \rightarrow \partial^{\theta} \Omega$ follows.
\end{proof}
\end{cor}
We have the following finiteness theorem on $\theta$-bumpy metric.
\begin{cor}
    On Riemannian $2$-disk $(M,\partial M,g)$ with strictly convex boundary and a $\theta$-bumpy metric, the set of all domains $D_{L}$ whose relative boundary is a capillary embedded geodesic whose boundary points are on $\partial M$ and $L^{\theta}$-functional achieves the value less than or equal to $L$ is finite.
    \begin{proof}
    Suppose $D_{L}$ is an infinite set. By Corollary A.2, there exists an infinite sequence of associated capillary varifolds $\{\partial^{\theta} \Omega_{k} \}$ converging to $\partial^{\theta} \Omega$ in $\Lambda_{\theta,L}$. This induces a nontrivial Jacobi field on $\partial^{\theta}\Omega$ and this contradicts to the bumpiness of the metric.
    \end{proof}
\end{cor}
\section{Regularity theorem of curve shortening flow with fixed boundary points}

In this section, we include the existence of regularity theorem of curve shortening flow with fixed boundary points. While it is stated and proven in Edelen \cite{E}, we include it here without the proof for completeness. The arguments of the proof relies on the distance comparison arguments which is initiated by Huisken \cite{H} (See also Theorem 2.6 in \cite{H}).
\begin{thm} [\cite{E}, Corollary 3.2]
Let $(M,\partial M,g)$ be a compact Riemannian surface with a strictly convex boundary. For $p,q \in \partial M$, let $F_{0}:[a,b] \rightarrow M$ be a smooth embedded curve where $F_{0}(a)=p$ and $F_{0}(b)=q$. Then the curve shortening flow for the initial curve $F_{0}$ with boundary conditions $F_{t}(a)=F_{0}(a), \, F_{t}(b) = F_{0}(b)$ has a smooth solution for all times $t>0$, which for $t \rightarrow \infty$ converges to a geodesic segment connecting $p$ and $q$.
\end{thm}
\section{Bumpy metric theorem of capillary embedded geodesics}
In this section, we prove the bumpy metric theorem of capillary embedded geodesics. The bumpy metric theorem asserts that capillary embedded geodesics do not have nontrivial Jacobi fields in a generic metric. We will mostly modify the bumpy metric theorem in embedded minimal hypersurface in \cite{W}. Our modifications mainly come from the free boundary version in \cite{ACS} and the stationary geodesic network version of \cite{CM} (See also \cite{St}). We outline the proof and mainly explain about the necessary modifications.

We state the bumpy metric theorem first:
\begin{thm}
   Let $(M,\partial M,g)$ be a compact Riemannian surface with a strictly convex boundary. For a $C^{\infty}$-generic set of a metric on $M$, no capillary embedded geodesic with contact angle $\theta$ on $M$ (whose boundary points are on $\partial M$) has a nontrivial Jacobi fields.
\end{thm}

For $q \ge 3$ a positive integer or $\infty$, denote $\Gamma^{q}$ to be the set of $C^{q}$ metrics on $M$ endowed with the smooth topology. For $3 \le j+1 \le q$, let us consider the space of embeddings $Emb(I, M)$ of $I= [0,1]$ to $M$. For an embedding $w \in Emb(I,M)$ we can take an equivalence class $[w_{\Omega}]$ associated with a connected domain $\Omega \subseteq M$ by
\begin{equation*}
    [w_{\Omega}] : = \{ (w \cdot \phi,\Omega) | \phi \in \text{Diff}(I) \text{ and } w = \partial^{i}\Omega \},
\end{equation*}
We consider the set $\mathcal{PE} (I,M)$ of equivalence classes $[w_{\Omega}]$ for an embedding $w$ and corresponding domain $\Omega$. We take a subset $S^{q}$ of $\mathcal{PE} = \mathcal{PE} (I,M)$ defined by
\begin{equation}
    \mathcal{S}^{q}:= \{ (\gamma, [w_{\Omega}]) \in \Gamma^{q} \times \mathcal{PE} | w \text{ is a }  \gamma-\text{embedded geodesic with a contact angle } \theta\},
\end{equation}
where $\theta \in (0, \pi/2)$ is a fixed angle and defined by the outer angle between $T_{z} \partial M$ and $T_{z} w(I)$ for $z \in w(I) \cap \partial M$ from $\Omega$. Let us endow the topology in $\mathcal{PE}$ induced by the Hausdorff distance between two associated domains of the embeddings, and endow the topology in $\mathcal{S}^{q}$ as an inherited topology from the product topology in $\Gamma^{q} \times \mathcal{PE}$. Let us denote $\pi^{q}_{\mathcal{S}}:\mathcal{S}^{q} \rightarrow \Gamma^{q}$ to be the projection map onto the first factor as \begin{equation}
    \pi^{q}_{\mathcal{S}}(\gamma, [w_{\Omega}] ) = \gamma.
\end{equation}

Given a Sard-Smale Theorem (Theorem C.2, Theorem 1.2 in \cite{Sm}) and a capillary version of the structure theorem (Theorem C.3, cf. Theorem 35 of \cite{ACS}) below, we can modify arguments in Section 7.1 of \cite{ACS} and apply Theorem C.2 to prove Theorem C.1. We prove Theorem C.3 in Subsection C.1.
\begin{thm} [\cite{Sm}, Theorem 1.2]
Let $\mathcal{M}$ and $\Gamma$ be a separable $C^{q}$-Banach manifolds. Suppose $\pi : \mathcal{M} \rightarrow \Gamma$ is a $C^{q}$ Fredholm map of Fredholm index $0$. Then the set of regular values for $\pi$ is a Baire generic set.
\end{thm} 
\begin{thm}\label{thm:St} [cf. \cite{ACS}, Theorem 35]
    Let $(M,\partial M,g)$ be a compact Riemannian surface with a strictly convex boundary. Then $\mathcal{S}^{q}$ is a separable Banach manifold of class $C^{q-1}$ and $\pi^{q}_{\mathcal{S}}$ is a $C^{q-1}$ Fredholm map of Fredholm index $0$. Furthermore, given any $(\gamma, [w_{\Omega}]) \in \mathcal{S}^{q}$, the nullity of $w_{\Omega}$ equals the dimension of the kernel of the linear map $D \pi^{q}_{\mathcal{S}}(\gamma, [w_{\Omega}])$. In particular, $w$ admits a non-trivial Jacobi field associated with $L^{\theta}$-functional if and only if the point $(\gamma, [w_{\Omega}])$ is critical for $\pi^{q}_{\mathcal{S}}$. 
\end{thm}
\subsection{The structure theorem} In this subsection, we prove Theorem C.3. We will outline the proof of Theorem C.3 and mostly emphasize the necessary change we need here. Let $V := I \times \mathbb{R}$ be a normal bundle over $I$ and $V_{I}:= T^{*}I \otimes V$. We take the convention that the positive direction of the normal bundle is chosen toward outside of an enveloping domain for each embedding of curves. Also denote $E:V \rightarrow M$ to be an embedding of $V$ to $M$ such that $E(\partial V) \subseteq \partial M$ with contact angle $\theta$ and each section of $V$ corresponds to the normal coordinate of the embedding in $M$. We consider a smooth map $l:\Gamma \rightarrow C^{q}(I \times V \times V_{I} ;\mathbb{R})$ where $\Gamma$ is an open subset of a Banach space $\mathcal{S}^{q}$. For $u \in C^{2}(I,V)$, we define a functional $L$ to be
\begin{equation}
    L(\gamma, u) = \int_{0}^{1} l(\gamma)(x,u,u'(x)) dx + \bigg[ b(\gamma)(x,u(x))\bigg]^{1}_{0}.
\end{equation}
Let us write $l_{\gamma}$ and $b_{\gamma}$ in place of $l(\gamma)$ and $b(\gamma)$. The first variational formula (see Proposition 41 of \cite{ACS}) gives the following:
\begin{align}
    \frac{d}{dt}\Big|_{t=0} L(\gamma, u+tv) &= \int_{0}^{1}(D_{2}l_{\gamma}(x,u,u')-(D_{3}l_{\gamma}(x,u,u'))_{x})v(x) dx \\
    &+ \bigg[(D_{3}l_{\gamma}(x,u,u')(x)+D_{2}b_{\gamma}(x,u))v(x)\bigg]_{0}^{1} \nonumber.
\end{align}
We define the first variation map $(H, \Theta): \Gamma \times C^{2}(I,V) \rightarrow C^{0}(I,V) \times C^{0}(\partial I,V)$ by 
\begin{align}
   H(\gamma,u) &= -D_{2}l_{\gamma}(x,u,u')-(D_{3}l_{\gamma}(x,u,u'))_{x}\\
   \Theta(\gamma,u) &= D_{3}l_{\gamma}(x,u,u')+D_{2}b_{\gamma}(x,u).
\end{align}
The first variation (49) gives the following directly.
\begin{prop} [cf. \cite{ACS}, Proposition 41 (1)] Suppose $L,H,\Theta$ is defined as in (48), (50), (51).  Then $H$ and $\Theta$ are $C^{q-2}$ banach maps.
\end{prop}
Now we consider the linearization of $(H,\Theta)$. The linearization $L_{H}: C^{2}(I, V) \rightarrow C^{0} (I,V)$ and $L_{\Theta}: C^{2}(I, V) \rightarrow C^{0}(\partial I,V)$ are defined by 
\begin{equation*}
    L_{H}(\gamma,u)[v] = \frac{d}{dt} \bigg|_{t=0} H(\gamma, u+tv),\,\, L_{\Theta}(\gamma,u)[v] = \frac{d}{dt} \bigg|_{t=0} \Theta(\gamma, u+tv)
\end{equation*}
and we define a Jacobi operator $L(\gamma,u): C^{2}(I, V) \rightarrow C^{0} (I,V) \times C^{0}(\partial I,V)$ by 
\begin{equation}
    L(\gamma,u)= (L_{H}(\gamma,u),L_{\Theta}(\gamma,u)).
\end{equation} 
Note that the formula of $L_{H}$ and $L_{\theta}$ is explicitly written by
\begin{align}
    L_{H}(\gamma,u)[v](\cdot) &= - (D_{22}l_{\gamma}(x,u,u')v + D_{23}l_{\gamma}(x,u,u')v')(\cdot) \\
   \nonumber &+  (D_{33}l_{\gamma}(x,u,u')v'(\cdot) + D_{23}l_{\gamma}(x,u,u')v(\cdot))_{x} \\
   \nonumber L_{\Theta}(\gamma,u)[v](\cdot) &= ((D_{23}l_{\gamma}(x,u,u')+D_{22}b_{\gamma}(x,u))v+ D_{33}l_{\gamma}(x,u,u')v')(\cdot)
\end{align}
\begin{prop} [cf. \cite{ACS}, Proposition 45 (2),(3)]
Suppose the Jacobi operator $L$ is given by (52) and $L_{H}$ is uniformly elliptic in the sense of
\begin{equation}
    D_{33}l_{\gamma}(x,u,u')>0
\end{equation}
for $x \in [0,1]$. Then
\begin{enumerate}
    \item Let us denote $\epsilon:\partial I \rightarrow I$ as the canonical inclusion. The operator $L$ is self-adjoint in the sense of
    \begin{equation*}
        (L(\gamma,u)[v], (w , w \circ \epsilon)) = (L(\gamma,u)[w], (v , v \circ \epsilon))
    \end{equation*}
    for any $v, w \in C^{2}(I,V)$.
    \item The operator $L$ is Fredholm of index $0$.
\end{enumerate}
\begin{proof}
    (1) follows from integration by parts. $L$ is a Fredholm operator by ODE uniqueness theorem and boundary conditions $L^{\Theta}$ on $\partial I$. Since $L$ is a self-adjoint operator and $L_{H}$ is uniformly ellptic in the sense of (54), the equation $(L_{H},L_{\Theta})(\phi,\psi) = (\xi,\eta)$ becomes a nondegenerate 2nd order ODE with Robin boundary condition by (53). Hence, the simpler ODE version of arguments in the proof of Proposition 45 in \cite{ACS} applies and we obtain that the Fredholm index is $0$.
\end{proof}
\end{prop}
The arguments in the proof of Proposition 46 in \cite{ACS} and we obtain the following structural result on the neighborhood of a fixed embedding. For a point $(\gamma_{0}, u_{0})$ with
\begin{equation*}
    H(\gamma_{0},u_{0}) = \Theta(\gamma_{0},u_{0}) =0,
\end{equation*}
we set $K = ker L(\gamma,u)$.
\begin{prop}[\cite{ACS}, Proposition 46]
Suppose $L$ is elliptic in the sense of (54). Assume that for every nonzero $\kappa \in K$ there exists a differentiable curve $\gamma : (-1,1) \rightarrow \Gamma$ with $\gamma(0) = \gamma_{0}$ and such that
\begin{equation*}
    \frac{\partial}{\partial s} \bigg|_{s=0} ((H(\gamma(s),u_{0}),\Theta(\gamma(s),u_{0})),(\kappa, \kappa \circ \epsilon)) \neq 0.
\end{equation*}
Then the map $(H, \Theta):\Gamma \times C^{2}(I,V)$ is a submersion near $(\gamma_{0},u_{0})$, so there exists a neighborhood $U_{\Gamma} \times U_{C^{2}(I,V)}$ of such point such that
\begin{equation*}
    \mathcal{M} = \{ (\gamma,u) \in U_{\Gamma} \times U_{C^{2}(I,V)} | H(\gamma,u) = 0, \Theta(\gamma,u) =0 \}
\end{equation*}
is a $C^{q-2}$-Banach submanifold with tangent space $ker D(H, \Theta)(\gamma,u)$ at each $(\gamma,u) \in \mathcal{M}$. Furthermore, the restricted projection operator $\Pi_{|\mathcal{M}}:\mathcal{M} \rightarrow \Gamma$, where $\Pi: \Gamma \times C^{2}(I,V) \rightarrow \Gamma$ is given by $\Pi(\gamma,u) = \gamma$, is a $C^{q-2}$-Fredholm map of index $0$.
\end{prop}
We denote a domain associated to the variation $u$ by $\Omega_{u}$. Let a signed boundary length function $b_{\gamma} : I \times \mathbb{R} \rightarrow \mathbb{R}$ which only depends on boundary values $u(0)$ and $u(1)$ to be
\begin{equation}
    b_{\gamma}(x,u(x)) = \cos \theta |\partial^{b} \Omega_{u}|.
\end{equation}

\begin{proof} [Proof of Theorem \ref{thm:St}] We take $l_{\gamma}$ to be the nonparametric length functional and a boundary length function in (55). Then $(\gamma, [\omega_{\Omega}]) \in \mathcal{S}^{q}$ if and only if a corresponding pair of an embedding $(\gamma,u)$ is stationary for the functional $L$.

We apply proposition C.5 and arguments in the proof of Theorem 35 in \cite{ACS} directly works and we obtain the structural theorem.
\end{proof}

\end{document}